\documentclass[nopreprintline]{elsarticle}
\usepackage{amssymb}
\usepackage{amsthm}
\usepackage{amsmath}
\usepackage{enumerate}
\usepackage{hyperref}
\usepackage{mathrsfs}
\usepackage{booktabs}

\theoremstyle{plain}

\newtheorem{thm}{Theorem}[section]
\newtheorem{cor}[thm]{Corollary}
\newtheorem{lem}[thm]{Lemma}
\newdefinition{defin}[thm]{Definition}
\newdefinition{rem}[thm]{Remark}

\newcommand{\FF}[1]{\mathbb F_{#1}}
\newcommand{\CC}{\mathbb C}
\newcommand{\NN}{\mathbb N}

\newcommand{\irr}{\textup{Irr}}

\newcommand{\sbs}{\subseteq}

\newcommand{\nor}{\vartriangleleft}

\newcommand{\ra}{\rightarrow}
\newcommand{\dg}{\delta (g)}
\newcommand{\pg}{\pi (g)}
\newcommand{\mb}[1]{ u_1, u_2,\ldots,  u_{#1}}
\newcommand{\diag}{\textup{diag}}
\newcommand{\s}{\scriptstyle}
\newcommand{\End}{\textup{End}}
\newcommand{\aut}{\textup{Aut}}
\newcommand{\out}{\textup{Out}}
\newcommand{\id}{\textup{id}}
\newcommand{\cP}{\mathscr P}
\newcommand{\ve}{\varepsilon}
\renewcommand{\t}{\times}

\makeatletter
    \def\ps@pprintTitle{%
      \let\@oddhead\@empty
      \let\@evenhead\@empty
      \def\@oddfoot{\reset@font\hfil\thepage\hfil}
      \let\@evenfoot\@oddfoot
}
\makeatother
\begin{document}
\begin{frontmatter}
\title{Every coprime linear group admits a base of size two}

\author[DE]{Zolt\'an Halasi\corref{cor1}\fnref{otka}} 
\ead{halasi.zoltan@renyi.mta.hu}
\author[BGF]{K\'aroly Podoski} 
\ead{podoski.karoly@pszfb.bgf.hu} 

\address[DE]{Institute of Mathematics, University of Debrecen, P.\ O.\ Box
  12, H-4010 Debrecen, Hungary} 
\address[BGF]{Budapest Business School, College of Finance and Accountancy, 
  Buzog\'any street 10-12, H-1149  Budapest, Hungary}

\fntext[otka]{Research partially supported by Hungarian National
  Research Fund (OTKA) under grant number K84233.}
\cortext[cor1]{Corresponding author}
\begin{abstract}
  Let $G$ be a linear group acting on the finite vector space $V$ and
  assume that $(|G|,|V|) =1$. In this paper we prove that $G$ has a
  base size at most two and this estimate is sharp. This generalizes
  and strengthens several former results concerning base sizes of
  coprime linear groups. As a direct consequence, we answer a question
  of I.~M.~Isaacs in the affirmative. Via large orbits this is related
  to the $k(GV)$ theorem.
\end{abstract}
\end{frontmatter}
\section{Introduction}
For a finite permutation group $H\leq \mathrm{Sym}(\Omega )$ a subset
$B=\{\omega_1,\dots, \omega_n\}\subseteq\Omega$ is called a base, if
its pointwise stabilizer in $G$ is the identity. There is a number of
algorithms for permutation groups related to the concept of base and
these algorithms are faster if the size of the base is small. Hence,
for both practical and theoretical reasons it is important to find
small bases for permutation groups. 
The minimal base size $b(H)$ of a permutation group $H$
is at least $\log |H|/\log |\Omega |$. Concerning upper bounds for
$b(H)$, a lot of results is proved if $H$ is solvable, if the action
of $H$ is primitive, or if $(|H|,|\Omega|)=1$.

For primitive groups it is asked by L.~Pyber \cite{pyber} that
the minimal base size is less than $C\log |H|/\log |\Omega |$ for some
universal constant $C$. Using the O'Nan-Scott theorem and CFSG, this question
has been reduced to the case when $H=V\rtimes G$ is an affine linear group 
such that $G\leq GL(V)$ acts irreducibly on the finite vector space $V$.

In case when $H$ is even solvable \'A.~Seress \cite{seress} proved
that the minimal base size is at most four. As the example of the full
affine linear group $AGL(V)$ shows, there is no absolute upper bound
for the minimal base size of affine groups.

A linear group $G\leq GL(V)$ is called coprime if $(|G|,|V|)=1$. 
If $G\leq GL(V)$ is such a group, then the affine group $H=V\rtimes G$ is
called a coprime affine group. It turns out that the minimal base size
of coprime affine groups is bounded. By a result of D.~Gluck and
K.~Magaard \cite{gluck_magaard} the minimal base size of any coprime
affine group is at most 95. As $b(H)=b(G)+1$ for an affine group $H$,
this means that $b(G)\leq 94$ for any coprime linear group.  Even for
coprime nilpotent linear groups $G\leq GL(V)$, regular orbits on $V$
do not always exist, so the best one can hope for is that $b(G)\leq 2$
for any coprime linear group $G\leq GL(V)$, that is, $C_G(x)\cap
C_G(y)=1$ for some $x,y\in V$.

In the past, the existence of such a pair of vectors was confirmed for
several special types of coprime linear groups such as for supersolvable groups
(T.~R.~Wolf \cite{wolf}), for groups of odd order (A.~Moreto and
T.~R.~Wolf \cite{moreto}) and for solvable groups (S.~Dolfi
\cite{dolfi_biz} and E.~P.~Vdovin \cite{vdovin_biz}). In our paper we
prove it without any additional assumption for $G$.  Our main result
is the following
\begin{thm}\label{thm:main_2base}
  Let $V$ be a finite vector space and $G\leq GL(V)$ be a coprime linear
  group. Then
  $b(G)\leq 2$ for this action, i.e. there exist $x,y\in V$ such that
  $C_G(x)\cap C_G(y)=1$.
\end{thm}
Using a lemma of Hartley and Turull \cite[Lemma 2.6.2]{hatu} one can
obtain a more general corollary
\begin{cor}\label{cor:1_2base}
  If $G$ is a group acting faithfully on a group $K$ and $(|G|,
  |K|)=1$ then there exist $x,y\in K$ such that $C_G(x)\cap C_G(y)=1$.
\end{cor}
  As a direct consequence of
Corollary \ref{cor:1_2base} we have another proof for a theorem of
P.~P.~P\'alfy and L.~Pyber.
\begin{thm}[P.~P.~P\'alfy, L.~Pyber {\cite[Theorem 1.]{pyberp3}}]
  If $G$ is a group acting faithfully on a group $K$ and $(|G|,
  |K|)=1$ then $|G|<|K|^2$.
\end{thm}
There is a general consequence of Corollary
\ref{cor:1_2base}, namely, in a faithful coprime action of a group
$G$, there always exists an orbit of size at least $\sqrt{|G|}$. This
answers a question raised by I.~M.~Isaacs.
\begin{cor}\label{cor:2_2base}
  If $G$ is a group acting faithfully on a group $K$ and $(|G|,
  |K|)=1$ then there exists $x\in K$ such that $|C_G(x)|\leq
  \sqrt{|G|}$.
\end{cor}
Note that this corollary answers a special case of an important question of
G.~Malle and G.~Navarro \cite[Question 10.1]{mana} (with $P=1$).

This paper is organized as follows. In Section 2 we investigate the
question that if $G$ is a permutation group acting on $\Omega$ then,
under some assumptions on $G$, at most how many parts must
$\Omega $ be partitioned into such that only the identity element of $G$
fixes every part of this partition. Using results of S.~Dolfi
\cite{dolfi_perm} and \'A.~Seress \cite{seress2}, the main result of
this section is Theorem \ref{thm:regpart_notalt}.  Using this theorem,
we finish Section 2 by reducing Theorem \ref{thm:main_2base} to
primitive linear groups. In the following three sections we solve the
problem for tensor product actions, for almost quasisimple groups and
for groups of symplectic type by using results and methods from the
papers of M.~W.~Liebeck and A.~Shalev \cite{LS}, C.~K\"ohler and
H.~Pahlings \cite{kopa} and an unpublished manuscript of the authors
\cite{arxiv_sajat}.  In the final section we prove Theorem
\ref{thm:main_2base} for arbitrary coprime linear groups by using an
induction argument and the results of the previous sections.
\section{Regular partitions for permutation groups}
\begin{defin}
  Let $\Omega$ be a finite set and let $G\leq S_\Omega$ be a
  permutation group on $\Omega$. A partition
  $\Omega=\cup_{i=1}^t\Omega_i$ is called $G$-regular, if
  $\cap_{i=1}^t G_{\Omega_i}=1$, where $G_X$ denotes the setwise
  stabilizer of $X$ for any $X\sbs\Omega$.  In our terminology, empty
  sets as parts of a partition are allowed.
\end{defin}
\begin{rem}
  In graph theory, this concept is sometimes referred to as
  distinguishing partition, and the minimal number of parts $t$ of a
  distinguishing partition is called the distinguishing number.
  In graph theory, the distinguishing number is just the minimal
  number of colors needed to color the vertices of a graph in such a
  way that the colored graph has trivial automorphism group.
\end{rem}
\begin{thm}\label{thm:regpart_notalt}
  Let $\Omega$ be a finite set and let $G\leq S_\Omega$ be a
  permutation group such that $G$ does not contain the alternating
  group $A_t$ as a section for some $t\geq 3$. Then there exists a
  $G$-regular partition of $\Omega$ with $t$ parts.
\end{thm}
\begin{proof}
  Our proof is based on a paper of S.\ Dolfi \cite{dolfi_perm}.
  Following the notation of Dolfi, for every $s\in \NN,\ s\geq 2$ let
  \[
  \cP_s(\Omega)=\left\{(\Lambda_1,\Lambda_2,\ldots,\Lambda_{s})
  \,|\,\Lambda_i\sbs \Omega,\ 
  \Lambda_i\cap \Lambda_j=\emptyset\textrm{ for }
  i\neq j,\ \cup_{i=1}^{s}\Lambda_i=\Omega\right\}
  \]
  be the set of (ordered) partitions of $\Omega$ into $s$ parts.  Then
  $G$ acts on $\cP_s(\Omega)$ in a natural way for any $s$ and we
  need to prove that $G$ has a regular orbit on $\cP_t(\Omega)$.
  We will use the following results proved by Dolfi: \textit{
    \begin{enumerate}
    \item If for some $t\in \NN$ every primitive component $(H,\Delta)$ of 
      $G$ has at least $t$ different regular orbits on $\cP_t(\Delta)$, 
      then $G$ has a regular orbit on $\cP_t(\Omega)$. 
      (see \cite[Theorem 2.]{dolfi_perm} with $\mu=\{\textrm{all primes}\}$)
    \item
      If $G$ has a regular orbit on $\cP_s(\Omega)$ for some $s<t$, then 
      $G$ has at least $t$ different regular orbits on $\cP_t(\Omega)$. 
      (use \cite[Remark 2.]{dolfi_perm} repeatedly) 
    \end{enumerate}
  } 
  By using \textit{1.}, we only need to prove that every primitive
  component $(H,\Delta)$ of $G$ has at least $t$ distinct regular
  orbits on $\cP_t(\Delta)$.  Noticing that $H$ has a regular
  orbit on $\cP_2(\Delta)$ if and only if there is a subset
  $X\sbs\Delta$ such that the setwise stabilizer $H_X$ of $X$ in $H$ is
  $1$ we can use a result of \'A.~Seress \cite[Theorem
  2.]{seress2}.  (see also \cite[Theorem 1.]{dolfi_perm}) In Seress'
  paper, the primitive permutation groups $(H,\Delta)$ satisfying
  $H\ngeq A_{|\Delta|}$ and having no regular orbit on $\cP_2(\Delta)$ have
  been determined using CFSG and the GAP system \cite{gap}. (We
  note that there is a small difference between the list of groups
  given by Seress and the list appearing in Dolfi's paper, since the
  latter assumes $H\ngeq A_{|\Delta|}$ only for $n\geq 5$. In the
  following we use Dolfi's list.) If $(H,\Delta)$ has a regular orbit
  on $\cP_2(\Delta)$, then it has at least $t$ distinct regular orbits
  on $\cP_t(\Delta)$ by using \textit{2.}. If $H\geq A_{|\Delta|}$
  but $H$ does not contain $A_t$ as a section, then $|\Delta|<t$, so
  $(H,\Delta)$ has a regular orbit on $\cP_s(\Delta)$ for
  $s=|\Delta|<t$, therefore, it has at least $t$ distinct regular
  orbits on $\cP_t(\Delta)$. Hence it remains to show that
  $(H,\Delta)$ has at least $t$ regular orbits on $\cP_t(\Omega)$
  in case $(H,\Delta)$ is one of the 46 groups listed in
  \cite[Theorem 1.]{dolfi_perm}. For these groups, the minimal value
  $s$ for which $H$ has a regular orbit on $\cP_s(\Omega)$, along with
  their multiplicities, have been calculated by Dolfi (see \cite[Lemma
    1. a,c]{dolfi_perm}). By using the GAP system \cite{gap} he proved
  that $H$ has at least three regular orbits on $\cP_3(\Delta)$ (hence
  it has at least $t$ regular orbits on $\cP_t(\Delta)$ for any
  $t\geq 3$) unless $(H,\Delta)$ is as in Table \ref{tab:pr_perm}).

  \begin{table}[ht]
    \begin{center}
      \begin{tabular}{@{}lcccc@{}}
        \toprule
        &&&\multicolumn{2}{c}{Number of regular orbits}\\
        \cmidrule(r){4-5}
        H	&$|\Delta|$&$\max\{t\,|\,A_t\,sec.\, H\}$&
        on $\cP_3(\Delta)$& on $\cP_4(\Delta)$\\
        \midrule
        $S_3$		&3&   3        &1&$\geq 4$\\
        $PSL(2,5)$	&6&   5        &1&$\geq 4$\\
        $P\Gamma L(2,8)$&9&   3        &1&$\geq 4$\\
        $S_4$		&4&   4        &-&\quad 1\\
        $PGL(2,5)$	&6&   5        &-&$\geq 4$\\
        $PSL(3,2)$	&7&   4        &-&$\geq 4$\\
        $M_{11}$	        &11&  6        &-&$\geq 4$\\
        $M_{12}$	        &12&  6        &-&$\geq 4$\\
        $ASL(3,2)$	&8&   4        &-&\quad 1\\
        \bottomrule
      \end{tabular}
    \end{center}
    \caption{Primitive groups with less than three regular orbits on
      $\cP_3(\Delta)$.\label{tab:pr_perm}}
  \end{table}

  For $t=3$, we get $(H,\Delta)$ cannot be any of these primitive
  permutation groups. For $t=4$, the only possibilities for $H$ are
  $S_3$ and $P\Gamma L(2,8)$ but in these cases there are at least $4$
  regular orbits on $\cP_4(\Delta)$. Finally, for $t\geq 5$ each of these primitive permutation groups has at least $t$ regular
  orbits on $\cP_t(\Delta)$.
\end{proof}
\begin{cor}\label{cor:dist_num}
  Let us assume that $G\leq S_\Omega$ with $t\nmid |G|$.
  Then there exists a $G$-regular partition of $\Omega$ with $t$ parts.
\end{cor}
\begin{proof}
  For $t=2$ we have that $G$ is a group of odd order, hence there is a
  $G$-regular partition of $\Omega$ with $2$ parts by a result of D.\
  Gluck \cite[Corollary 1.]{gluck}.  Otherwise, if $t\nmid |G|$, then
  $G$ cannot contain $A_t$ as a section, so there is a $G$-regular
  partition of $\Omega$ with $t$ parts by Theorem
  \ref{thm:regpart_notalt}
\end{proof}
\begin{rem}
  Without using CFSG, it can be shown that if $G$ is a permutation group on
  $\Omega$ such that no $g\in G$ contains a cycle of length
  greater than $t$, then there is a $G$-regular partition of $\Omega$
  with $t$ parts.
\end{rem}
By using Corollary \ref{cor:dist_num}, we close this section by reducing
Theorem \ref{thm:main_2base} to primitive linear groups.
\begin{thm}\label{thm:imprim}
  Let us assume that $b(H)\leq 2$ for any finite vector space $W$ and for any 
  coprime primitive linear group $H\leq GL(W)$. Then $b(G)\leq 2$ for any 
  finite vector space $V$ and for any coprime linear group $G\leq GL(V)$.
\end{thm}
\begin{proof}
  Let $G\leq GL(V)$ be any coprime imprimitive linear group acting on
  $V$, where $V$ is a finite vector space over the finite field $K$ of
  characteristic $p$.  By Maschke's theorem, $V$ is completely reducible
  as a $KG$-module.  If $V$ is not irreducible as a $KG$-module, then
  $V=V_1\oplus V_2$ for some proper $G$-invariant subspaces
  $V_1,V_2\leq V$.  By using induction on $\dim V$, for $i=1,2$ set
  $x_i,y_i\in V_i$ such that $C_G(x_i)\cap C_G(y_i)=C_G(V_i)$. Then
  $C_G(x_1+x_2)\cap C_G(y_1+y_2)= C_G(V_1)\cap C_G(V_2)=1$.

  In the following let $G\leq GL(V)$ be an irreducible, imprimitive
  linear group.  Thus, there is a decomposition $V=\oplus_{i=1}^k V_i$
  such that $k\geq 2$ and $G$ permutes the subspaces $V_i$ in a
  transitive way. We can assume that the decomposition cannot be
  refined.  For each $1\leq i\leq k$ let $H_i=\{g\in G\;|\;
  gV_i=V_i\}$ be the stabilizer of $V_i$ in $G$.  Then
  $H_i/C_{H_i}(V_i)\leq GL(V_i)$ is a primitive linear group, and the
  subgroups $H_i$ are conjugate in $G$.  Of course, $(|H_1|,|V_1|)=1$,
  so, by using the assumption, we can find vectors $x_1,y_1\in V_1$
  such that $C_{H_1}(x_1)\cap C_{H_1}(y_1)=C_{H_1}(V_1)$.  Let
  $\{g_1=1, g_2,\ldots,g_k\}$ be a set of left coset representatives
  for $H_1$ in $G$ such that $V_i=g_iV_1$ for all $1\leq i\leq k$ and
  let $x_i=g_ix_1$, $y_i=g_iy_1$. It is clear that
  $H_i=H_1^{g_i^{-1}}$ and $C_{H_i}(x_i)\cap
  C_{H_i}(y_i)=C_{H_i}(V_i)$.

  Now, $N=\cap_{i=1}^k H_i$ is a normal subgroup of $G$, the quotient
  group $G/N$ acts faithfully and transitively on the set
  $\Omega=\{V_1,V_2,\ldots,V_k\}$, and $|G/N|$ is coprime to $p$.
  Using Corollary \ref{cor:dist_num}, there is a $G/N$-regular
  partition of $\Omega$ into $p$ parts, say,
  $\Omega=\Lambda_1\cup\ldots\cup\Lambda_p$. Then we can choose a
  vector $(a_1,a_2,\ldots, a_k)\in \FF p^k$ such that $a_i=a_j$ if and
  only if $V_i$ and $V_j$ are in the same part of the partition
  $\Omega=\Lambda_1\cup\ldots\cup\Lambda_p$.  Now, let the vectors
  $x,y\in V$ be defined as
  \[
  x=\sum_{i=1}^k x_i,\qquad\qquad y=\sum_{i=1}^k (y_i+a_ix_i).
  \]
  We claim that $C_G(x)\cap C_G(y)=1$.  Let $g\in C_G(x)\cap
  C_G(y)$. Assuming that $gV_i=V_j$ for some $1\leq i,j\leq k$ we get
  $gx_i=x_j$ and $g(y_i+a_ix_i)=(y_j+a_jx_j)$.  Let
  $g'=g_j^{-1}gg_i\in G$. Then
  \begin{equation}\label{eq:imprim:unitri} 
    g'x_1=x_1 \textrm{\quad and\quad} g'(y_1+a_ix_1)=(y_1+a_ix_1)+(a_j-a_i)x_1,
  \end{equation}
  so $g'$ stabilizes the subspace $\left<x_1,y_1\right>\leq V_1$.  If
  $y_1=cx_1$ for some $c\in \FF p$, then $g'y_1=y_1$.  Using Equation
  (\ref{eq:imprim:unitri}) we get $a_j=a_i$. Otherwise,
  $x_1,y_1+a_ix_1$ form a basis of the $\langle g'\rangle$-invariant
  subspace $\langle x_1,y_1\rangle$.  With respect to this basis the
  restriction of $g'$ to this subspace has matrix form
  \[
  \begin{pmatrix}1&a_j-a_i\\ 0&1\\ \end{pmatrix}.
  \]
  If $a_j-a_i\neq 0$, then this matrix has order $p$, so $p$ divides
  the order of $g'\in G$, a contradiction. Hence in any case $a_i=a_j$
  holds for $gV_i=V_j$, which exactly means that $gN\in G/N$
  stabilizes the vector $(a_1,a_2,\ldots,a_k)$.  It follows that $g\in
  N$. So $gx_i=x_i$ and $gy_i=y_i$ holds for any $1\leq i\leq k$, and
  $g\in\cap_{i=1}^k C_{H_i}(V_i)=C_G(V)=1$ follows.
\end{proof}
\section{Tensor product actions}
The purpose of this section is to find a base for linear groups acting
on tensor product spaces. Throughout this section every vector space
is assumed to be finite dimensional over the finite field $\FF q$ for
some prime power $q$. The group of scalar transformations will be
denoted by $Z$.

We start this section with the following definition introduced in \cite{LS}.
\begin{defin}
  For a linear group $G\leq GL(V)$ a strong base for $G$ is a set
  $B\sbs V$ of vectors such that any element $g\in G$ fixing $\langle
  v\rangle$ for every $v\in B$ is a scalar matrix.  The minimal size
  of a strong base for $G$ is denoted by $b^*(G)$.
\end{defin}
\begin{rem} 
  It was proved by Liebeck and Shalev \cite[Lemma 3.1]{LS} that
  $b(G)\leq b^*(G)\leq b(G)+1$ holds for any linear group $G\leq
  GL(V)$.
\end{rem}
For the base sizes and strong base sizes of coprime linear
groups we have the following.
\begin{lem}\label{lem:strong}
  Let $G\leq GL(V)$ be a coprime linear group containing $Z$ with a
  two-element base $u_1,u_2\in V$. Then $u_1,\ u_1+\gamma u_2$ is a
  strong base for some $\gamma\in \FF q$.
\end{lem}
\begin{proof}
  The case when $u_1$ and $u_2$ are linearly dependent is trivial,
  since a one-element base for a linear group containing $Z$ is always
  a strong base for $G$. Now, let $u_1,\ u_2$ be linearly independent
  vectors. Then $u_1, u_1+\alpha u_2$ is also a base for $G$ for every
  $\alpha\in \FF q^\t$.  For $\alpha\in \FF q^\t$ let
  \[X_\alpha=\{\lambda\in \FF q^\t\setminus\{1\}\,|\, \exists\, g\in
  G; gu_1=u_1,\,g(u_1+\alpha u_2)=\lambda(u_1+\alpha u_2)\}.\] We
  claim that $X_\alpha\cap X_\beta=\emptyset$ if $\alpha\neq \beta$
  are non-zero field elements.  Indeed, let us assume that $\lambda\in
  X_\alpha\cap X_\beta$. Then there exist $g,h\in G$ such that
  \[
  \begin{array}{r@{\;=\;}l}
    gu_1&u_1\\
    g(u_1+\alpha u_2)&\lambda(u_1+\alpha u_2)
  \end{array}
  \quad\textrm{and}\quad
  \begin{array}{r@{\;=\;}l}
  hu_1&u_1\\
  h(u_1+\beta u_2)&\lambda(u_1+\beta u_2)
  \end{array}
  .
  \]
  Then both $g$ and $h$ fixes the two dimensional subspace $U=\langle
  u_1,u_2\rangle$. The restrictions of $g$ and $h$ to this subspace
  have the following matrix forms:
  \[
  g_U=\begin{pmatrix}
    1&(\lambda-1)/\alpha\\
    0&\lambda
  \end{pmatrix}
  \qquad\textrm{and}\qquad
  h_U=\begin{pmatrix}
    1&(\lambda-1)/\beta\\
    0&\lambda
  \end{pmatrix}.
  \]
  Thus, $g_U^{-1}h_U$ is an upper unitriangular matrix different from $1_U$,
  hence $p=o(g_U^{-1}h_U)\mid o(g^{-1}h)$ for $g^{-1}h\in G$, a contradiction.

  Using the pigeonhole principle, it follows that
  $X_{\gamma}=\emptyset$ for some $\gamma\in \FF q^\t$.  Let
  $v_1,v_2\in V$ be defined as $v_1=u_1$ and $v_2=u_1+\gamma u_2$.
  Let us assume that $g\in G$ fixes both $\langle v_1\rangle$ and
  $\langle v_2\rangle$ and let $\lambda_1,\lambda_2\in \FF q^\t$ be such
  that $gv_1=\lambda_1v_1,\ gv_2=\lambda_2v_2$. As $G$ contains every
  non-zero scalar matrix, $g_0=\lambda_1^{-1}g\in G$. Now, $g_0$ fixes
  $v_1$ and moves $v_2$ by a scalar multiple.  Using the definition of
  $\gamma$, we get $g_0v_2=v_2$. As $v_1, v_2$ is also a base for $G$,
  it follows that $g_0=1$, that is, $g$ is a scalar matrix.  Hence
  $v_1,v_2$ is a strong base for $G$.
\end{proof}
The following lemma is probably well-known, but we have not found any reference.
\begin{lem}\label{lem:tensor_direct}
  Let $G$ be a finite group, let $K$ be a field of characteristic
  $p>0$ and let $V$ be an absolutely irreducible $KG$-module of dimension $n$.
  If $G$ is a direct or a central product of the
  groups $L_1,L_2,\ldots, L_t$, then $V=\otimes_{i=1}^t V_i$ for some
  absolutely irreducible $KL_i$-modules $V_i$.
\end{lem}
\begin{proof}
  Clearly we can assume that $t=2$ and $G=L_1\times L_2$ is a direct
  product.  By Clifford theory, $V$ is completely reducible as a
  $KL_1$-module, and $L_2\sbs C_G(L_1)$ permutes the homogeneous
  components of this $KL_1$-module in a transitive way.  It follows
  that $V=U_1\oplus\ldots\oplus U_k$ where $U_1,\ldots,U_k$ are
  isomorphic irreducible $KL_1$-modules.  Now, let $T>K$ be the
  algebraic closure of $K$. By extending scalars, we can make $V$ into
  an irreducible $TG$-module, denoted by $V^T$. Again, $V^T$ is a
  direct sum of some isomorphic (absolutely) irreducible
  $TL_1$-modules. On the other hand, $V^T=U_1^T\oplus\ldots\oplus
  U_k^T$, so each $U_i^T$ decomposes into the direct sum of isomorphic
  irreducible $TL_1$-modules. Comparing this result with \cite[Theorem
  9.21/a,b]{Ibook} it follows that every $U_i^T$ is itself
  irreducible as a $TL_1$-module, hence $V=U_1\oplus\ldots\oplus U_k$
  is the decomposition of $V$ into isomorphic absolutely irreducible
  $KL_1$-modules.  With respect to a suitable basis of $V$, for the
  matrix representation of $X:G\ra GL(n,K)$ and for every $g_1\in L_1$
  we have $X(g_1)=X_1(g_1)\otimes I_k$, where $X_1:L_1\ra GL(n/k,K)$
  is an absolutely irreducible representation of $L_1$. Using Schur's
  Lemma, and \cite[Theorem 9.2/c]{Ibook}, it follows that $X(L_2)\leq
  C_{GL(n,K)}(X(L_1))=\{I_{n/k}\otimes B\,|\,B\in
  GL(k,K)\}$. Therefore, $X(g_2)=I_{n/k}\otimes X_2(g_2)$ for some
  (absolutely) irreducible representation $X_2:L_2\ra GL(k,K)$ of
  $L_2$. Thus, $X=X_1\otimes X_2$ and the result follows.
\end{proof}
In the following let $V_1$ be an $m\geq 2$ dimensional vector space over 
the field $\FF q$, and let $\{x_1,x_2,\ldots,x_m\}\sbs V_1$ be a basis of $V_1$. 
For any $k$ and $1\leq i\leq m$ let the vector space $V_1^{(k)}$ and 
$x_i^{(k)}\in V_1^{(k)}$ be defined as
\[
V_1^{(k)}:=\underbrace{V_1\otimes\cdots\otimes V_1}_{k \textrm{ factors}},
\qquad
x_i^{(k)}:=\underbrace{x_i\otimes\cdots\otimes x_i}_{k \textrm{ factors}}.
\]
With this notation we have the following.
\begin{lem}\label{lem:tensor_x}
  Let $V_1$ be a vector space over $\FF q$ with basis
  $\{x_1,\ldots,x_m\}\sbs V_1$, and let $t\geq 2$.  Then the group
  $B=\underbrace{GL(V_1)\otimes\cdots\otimes GL(V_1)}_{t \textrm{
      factors}}$ acts in the natural way on $V=V_1^{(t)}$. Let
  $x=\sum_{i=1}^m x_i^{(t)}\in V$. Then the matrix form of $C_B(x)$
  with respect to the basis $\{u_1\otimes\cdots\otimes
  u_t\,|\,u_i\in\{x_1,\ldots,x_m\}\ \forall 1\leq i\leq t\}$ is the following.
  \begin{enumerate}
  \item
    For $t=2$ we have $C_B(x)=\{A\otimes A^{-T}\,|\,A\in GL(m,q)\}$.
  \item For $t\geq 3$ let $P\leq GL(m,q)$ be the group of permutation
    matrices and $D\leq GL(m,q)$ be the group of diagonal matrices
    with respect to the basis $\{x_1,\ldots, x_m\}$. Then
    \[
    C_B(x)=\bigg\{A_1S\otimes \cdots\otimes A_tS\,|\, S\in P,\
    A_1,\ldots,A_t\in D,\ A_1\dotsm A_t=I\bigg\}.
    \]
  \end{enumerate}
\end{lem}
\begin{proof} 
  For $t=2$ let $A_1=(\alpha_{ij}),\ A_2=(\beta_{ij})\in GL(m,q)$ with
  $A_1\otimes A_2\in C_B(x)$. As the coefficient of $x_i\otimes x_j$
  in $(A_1\otimes A_2)(x_k\otimes x_k)$ is $\alpha_{ik}\beta_{jk}$, it
  follows that
  \[
  \sum_{k=1}^m(A_1\otimes A_2)(x_k\otimes x_k)= \sum_{k=1}^m (x_k\otimes x_k)\iff
  \sum_{k=1}^m \alpha_{ik}\beta_{jk}=\delta_{ij}\textrm{ for all } 1\leq i,j\leq m.
  \]
  This exactly means that $A_1A_2^T=I$, which proves part 1.

  In case $t\geq 3$ let $M_1\otimes\cdots\otimes M_t\in C_B(x)$ for some
  $M_1,M_2,\ldots,M_t\in GL(m,q)$.  By using \cite[Lemma 3.3 (i)]{LS}
  for 
  \[
  W_2:=\langle x_1^{(2)},\ldots, x_m^{(2)}\rangle,\quad
  W_{t-2}:=\langle x_1^{(t-2)},\ldots,x_m^{(t-2)}\rangle
  \]
  and for the decomposition $x=\sum_{i=1}^m x_i^{(2)}\otimes
  x_i^{(t-2)}$ we get that $(M_1\otimes M_2)(x_j^{(2)})\in W_2$ for
  every $1\leq j\leq m$. In this product the coefficient of
  $x_k\otimes x_l$ is $\alpha_{kj}\beta_{lj}$ where
  $M_1=(\alpha_{ij}),\ M_2=(\beta_{ij})$. So $\alpha_{kj}\beta_{lj}=0$
  unless $k=l$. If $\alpha_{kj}\neq 0$ for some $k$ then
  $\beta_{lj}=0$ for every $l\neq k$ and $\beta_{kj}\neq 0$. Reversing
  the role of $\alpha$ and $\beta$ and applying this argument for
  every $1\leq j\leq m$, it follows that both $M_1$ and $M_2$ are
  monomial matrices with the same permutation part. Of course, all of
  this can be said for every pair of matrices $M_i,\ M_j$ with $1\leq
  i\neq j\leq t$, which proves that
  \[
  C_B(x)\sbs\bigg\{A_1S\otimes \cdots\otimes A_tS\,|\, S\in P,\
  A_1,\ldots,A_t\in D\bigg\}.
  \]
  If $M=A_1S\otimes \cdots\otimes A_tS\in C_B(x)$ for some $S\in P$
  and $A_1,\ldots,A_t\in D$, then the $i$-th entry of the main
  diagonal of $A_1\dotsm A_t\in D$ is the coefficient of $x_i^{(t)}$
  in $Mx$, which is 1.  This proves that $C_B(x)$ is included in the
  set given in part 2 of the statement.  The proof of the converse
  containment is an obvious calculation.
\end{proof}
\begin{rem} 
  We note that in case $t=2$ this result was essentially proved in the proof of
  \cite[Lemma 3.3 (iii)]{LS}, while in case $t\geq 3$ there is a wrong
  description of this centralizer in the proof of \cite[Lemma
  3.5]{LS}.
\end{rem}
\begin{thm}\label{thm:tensor_power}
  Let $V_1$ be a vector space over $\FF q$ of dimension $d\geq 2$ and
  let $Z\leq G_1\leq GL(V_1)$ be any linear group such that
  $b^*(G_1)\leq 2$. For any $t\geq 2$ let $G=G_1\wr_c S_t$ be the
  central wreath product of $G_1$ by $S_t$, that is, $G$ is a split
  extension of the base group
  $B=\underbrace{G_1\otimes\cdots\otimes G_1}_{t \textrm{
      factors}}$ by $S_t$. Then $G$ acts faithfully on the tensor
  power $V=V^{(t)}$ in a natural way, so it is embedded into
  $GL(V)$. If $(d,t)\neq (2,2)$ then $b(G)\leq 2$ holds for this action.
\end{thm}
\begin{proof}
  Let $x_1,y_1\in V_1$ be a strong base for $G_1$ such that $x_1,y_1$
  are linearly independent and let $U_1=\langle x_1,y_1\rangle$.  Let
  $x=x^{(t)}+y^{(t)}$. As $S_t$ acts on the tensor product by
  permuting factors, we clearly have $S_t\leq C_G(x)$, hence
  $C_G(x)=H\rtimes S_t$ for some subgroup $H$ of the base group
  $B$. Let $h_1\otimes \cdots\otimes h_t\in H$ be any element of $H$.
  (Note that $h_1,\ldots, h_t\in G_1$ are only defined up to scalars.)
  Applying \cite[Lemma 3.3(i)]{LS} for $V_1$ and $V^{(t-1)}$ it
  follows that $U_1\leq V_1$ is fixed by $h_1\in G_1$. Using that
  $S_t$ acts by conjugation on $H$ and it permutes its coordinates in
  a transitive way we get that $U_1$ is an invariant subspace of $h_i$
  for any $1\leq i\leq t$, so $U=U^{(t)}\leq V^{(t)}$ is an invariant subspace of
  $C_G(x)$.  By Lemma \ref{lem:tensor_x}, the matrix form of $H_{U}$ is
  the following.\par
  In case $t=2$ we have $H_{U}\sbs \{A\otimes A^{-T}\,|\,A\in
  GL(2,q)\}$, while in case $t\geq 3$
  \[
  H_{U}\sbs\bigg\{A_1S^\varepsilon\otimes\cdots\otimes A_tS^\varepsilon\,
  \left|\right.\,\varepsilon\in \{0,1\}, A_i=
  \begin{pmatrix}\lambda_i&0\\0&\mu_i\end{pmatrix}, 
  \prod_i\lambda_i=\prod_i\mu_i=1\bigg\},
  \]
  where $S=\begin{pmatrix}0&1\\1&0\end{pmatrix}\in GL(2,q)$.\par
  Now, we define $y\in V^{(t)}$ such that $C_G(x)\cap C_G(y)=1$.
  First, let us assume that $t=2$. Using that $\dim V_1=d\geq 3$ by
  our assumption, let $z_1\in V_1\setminus U_1$. Let $y\in V^{(2)}$ be
  defined as
  \[
  y=x_1\otimes z_1+y_1\otimes x_1.
  \]
  Let us assume that $g\in C_G(x)\cap C_G(y)$. Then $g=h\sigma$, where
  $h=h_1\otimes h_2\in H$ for some $h_1,h_2\in G_1$ and $\sigma\in
  S_2$.  Furthermore, $h$ fixes $U=U_1\otimes U_1$ and $h_{U}=A\otimes
  A^{-T}$ for some $A\in GL(2,q)$.  If $\sigma\neq 1$, then $gy\in
  V_1\otimes U_1$, while $y\notin V_1\otimes U_1$, a contradiction.
  Thus, $g=h$ and $gy\in Ax_1\otimes V_1+Ay_1\otimes
  A^{-T}x_1$. Extending $x_1\otimes z_1, y_1\otimes x_1, y_1\otimes
  z_1, y_1\otimes y_1$ to a basis of $V$ and writing $gy$ as a linear
  combination of the basis elements we get that $g(x_1\otimes z_1)$
  contains $x_1\otimes z_1$, but it does not contain $y_1\otimes
  z_1$. It follows that $x_1$ is an eigenvector of $A$.  Now, using
  that $Ay_1\otimes A^{-T}x_1$ contains $y_1\otimes x_1$ but it does
  not contain $y_1\otimes y_1$ we get that $x_1$ is an eigenvector of
  $A^{-T}$ too. It follows that $A$ is a diagonal matrix. Since
  $h_{1,U_1}=A, h_{2,U_1}=A^{-T}$ and we chose $x_1,y_1$ to be a
  strong base for $G_1$, it follows that $h_1,h_2\in Z$. Hence
  $g=h_1\otimes h_2=1$.

  For $t\geq 3$ let $y$ be defined as
  \[
  y=x^{(t)}+\sum_{i=1}^{t-1}x^{(t-i)}\otimes y^{(i)}.
  \]
  Now, if $g=h\sigma\in C_G(x)\cap C_G(y)$, where $h=h_1\otimes
  h_2\otimes\cdots\otimes h_t\in H$ for some $h_1,\ldots,h_t\in G_1$
  and $\sigma\in S_t$ then $h$ leaves invariant $U$ and
  $h_{U}=A_1S^\varepsilon\otimes\cdots\otimes A_tS^\varepsilon$.  Now,
  if $\varepsilon=1$, then $y$ contains $x^{(t)}$ with non-zero
  coefficient, while $gy$ does not, a contradiction. It follows that
  $h_{U}=A_1\otimes \cdots\otimes A_t$.

  As $h_{i,U_1}=A_i$ is a diagonal matrix with respect to the basis
  $x_1,y_1\in U_1$ and $x_1,y_1$ is a strong base for $G_1\leq
  GL(V_1)$ we get that $h_i\in Z$ for every $1\leq i\leq t$. It follows
  that $h=1$ and $g=\sigma\in S_t$. If $\sigma(i)=j$ for some $i<j$,
  then $y$ contains $x^{(i)}\otimes y^{(t-i)}$ with non-zero
  coefficient, while $gy$ does not, which proves that $g=1$, as
  claimed.
\end{proof}
\begin{rem}
  The assumption $(d,t)\neq (2,2)$ is crucial in the previous theorem.
  By taking $V_1\simeq \FF 3^2$ and $G_1=SL(V_1)\geq Z$ we get
  $b^*(G_1)=b(G_1)=2$, while the base size of $G=(G_1\otimes
  G_1)\rtimes S_2$ acting on $V_1\otimes V_1$ is $3$.
\end{rem}
\begin{thm}\label{thm:tensor_main}
  Let $V_1$ be a vector space over $\FF q$ of dimension $d\geq 2$ and
  let $Z\leq G_1\leq GL(V_1)$ be a coprime linear group with
  $b(G_1)\leq 2$.  Let $G=G_1\wr_c S_t$ be the central wreath product
  of $G_1$ by $S_t$ acting on $V=V_1^{(t)}$. Then $b(G)\leq 2$ for this
  action.
\end{thm}
\begin{proof}
  By Lemma \ref{lem:strong}, we have $b^*(G_1)\leq 2$. Therefore, we can
  apply Theorem \ref{thm:tensor_power} to conclude that $b(G)\leq 2$ unless
  $(d,t)=(2,2)$.\par
  In case $(d,t)=(2,2)$ let $\{x_1,x_2\}\sbs V_1$ be a strong base of
  $V_1$, let $\sigma\in S_2,\ \sigma\neq 1$ and let $\alpha\in\FF
  q^\t$ be a generator of the multiplicative group of $\FF q$. The
  case $2|q$ is trivial by \cite[Lemma 3.3 (ii)]{LS}, so we may
  assume that $q$ is odd. We assume by way of contradiction that
  there is no $u,v\in V$ such that $C_G(u)\cap C_G(v)=1$.  Let
  $u_i=x_i\otimes x_i,\ v_i=x_i\otimes x_{3-i}+\alpha x_{3-i}\otimes
  x_i$ for $i\in\{1,2\}$. An easy calculation shows that there exist
  $c_1,c_2\in \FF q^\t$ such that
  \begin{align*}
  1\neq g_1\in C_G(u_1)\cap C_G(v_1)&\iff g_1=(a_1\otimes a_1^{-1})\sigma 
  \textrm{ where } a_1=\begin{pmatrix}1&c_1\\0&\alpha\end{pmatrix},\\
  1\neq g_2\in C_G(u_2)\cap C_G(v_2)&\iff g_2=(a_2\otimes a_2^{-1})\sigma 
  \textrm{ where } a_2=\begin{pmatrix}1&0\\c_2&\alpha\end{pmatrix}.\\
  \end{align*}
  Now, $H=\langle a_1,a_2, Z\rangle \leq G_1$. Defining $H_1=\langle
  a_1,Z\rangle$ and $H_2=\langle a_2\rangle$ we get $|H_1|=(q-1)^2,\
  |H_2|=q-1$ and $H_1\cap H_2=1$. Therefore, $|H|\geq
  |H_1H_2|=|H_1||H_2|=(q-1)^3$. As $(|H|,q)=1$, and $|H|\mid
  |GL(2,q)|=q(q+1)(q-1)^2$, it follows that $H$ is a $q'$-Hall
  subgroup of $GL(V_1)$, in particular $|GL(V_1):H|=q$.  This results
  that $PSL(2,q)$ contains a subgroup of index $q$ which is impossible
  by \cite[Table 5.2.A, p. 175.]{Kleidman} unless $q\in
  \{3,5,7,9,11\}$. The remaining cases were checked using GAP \cite{gap}.
\end{proof}
\section{Almost quasisimple groups}\label{sec:quasisimple} 
In this section we suppose that $p$ is a prime, $V$ is an
$n$-dimensional vector space over $\FF p$ and $G\leq GL(V)$ is a
$p'$-group having a quasisimple irreducible normal subgroup $N$.  In
order to prove the existence of a base of size two for such a group,
we use a result of C. K\"ohler and H. Pahlings \cite{kopa}.  By using
D. P. M. Goodwin's theorem \cite{goodwin}, they proved that there
exists a regular orbit for such a group in most cases and they gave a
description on the exceptions (see Table \ref{tab:regorb}) with the
isomorphism type of a stabilizer of a vector having smallest order.
Using this result, for these linear groups we choose an $x\in V$ such
that $|C_G(x)|$ is as small as possible and we prove the existence of
a vector $y\in V$ such that $C_G(x)\cap C_G(y)=1$.
\begin{thm}[D. P. M. Goodwin \cite{goodwin}, C. K\"ohler and 
    H. Pahlings \cite{kopa} ]\label{thm:ghp} 
  Suppose $p$ is a prime,
  $V$ is an $n$-dimensional vector space over $\FF p$ and $G\leq
  GL(V)$ is a $p'$-group having a quasisimple normal subgroup $N$
  which is irreducible on $V$.  If $G$ has no regular orbit on the
  vectors of $V$, then one of the following holds:
  \begin{enumerate}
  \item $N =A_c$, the alternating group of degree $c$ for $c<p$ and $V$
    is the deleted permutation module for $N$ of dimension $c-1$.
  \item $G, n, p$ are as in Table \ref{tab:regorb}
  \begin{table}[ht]
  \begin{center}
  \begin{tabular}{@{}llll@{}}
  \toprule
  $G$ &$n$ &$p$ & minimal stabilizer\\
  \midrule
  $A_5\times Z$&$3$ &$11$ & $C_2$ \\
  $A_5.2\times Z$&$4$ &$7$ & $C_2$ \\
  $2.A_5\star Z$& $2$ &$29,41,61,11,19,31$ & $C_2, C_2, C_2,C_5, C_3, C_3$\\
  $Z.(8\star 2.A_5).2$&$4$ &$7$ & $V_4$ \\
  \hline
  $A_6.2\times Z$&$5$ &$7$ & $C_2$ \\
  $2.A_6.2\star Z$&$4$ &$7$ & $C_3$ \\
  $3.A_6\star Z$&$3$ &$19,31$ & $C_2, C_2$ \\
  \hline
  $2.A_7\star Z$&$4$ &$11$ & $C_3$ \\
  \hline
  $L_2(7)\times Z$&$3$ &$11$ & $C_2$ \\
  $Z.(6\times L_2(7)).2$&$6$ &$5$ & $C_2$ \\
  \hline
  $U_3(3)\times Z$&$7$ &$5$ & $C_2$ \\
  $U_3(3).2\times Z$&$7$ &$5$ & $C_2$ \\
  $(U_3(3)\times Z).2$&$6$ &$5$ & $S_3$ \\
  \hline
  $U_4(2)\times Z$&$5$ &$7,13,19$ & $S_4,V_4,C_2$ \\
  $U_4(2).2\times Z$&$6$ &$7,11,13$ & $D_{12},V_4,C_2$ \\
  $2.U_4(2)\star Z$&$4$ &$7,13,19,31,37$& $U_{72},U_{18},(C_3^2, C_9),C_3,C_2$ \\
  \hline
  $6_1.U_4(3).2_2\star Z$&$6$ &$13,19,31,37$ & $W(B_3),S_3\times C_2,V_4,C_2$ \\
  \hline
  $U_5(2)\times Z$&$10$ &$7$ & $V_4$ \\
  \hline
  $Sp_6(2)\times Z$&$7$ &$11,13,17,19$ & $C_2^3,V_4, C_2, C_2$ \\
  \hline
  $2.\Omega_8^+(2)\star Z$&$8$ &$11,13,17,19,23$ & $W(B_3), 
  S_4, S_3, V_4, C_2$ \\
  \hline
  $2.J_2\star Z$&$6$ &$11$ & $S_3$ \\
  \bottomrule
  \end{tabular}
  \caption{Coprime linear groups of quasisimple type with no regular
      orbit\label{tab:regorb}}
  \end{center}
  \end{table}
  
  In this table in the column headed by $G$ always the largest
  possible group is listed and $Z\simeq (\FF p)^\times$. In the column
  headed by ``minimal stabilizer'' the isomorphism type of a
  stabilizer of a vector having smallest order is displayed. It is
  always unique except for $2.U_4(2)\star Z$ in dimension $4$ for
  $p=19$, where there are two such types. $U_{72}, U_{18}$ are certain
  subgroups of $U_4(2)$ of order $72$ and $18$ respectively.
  \end{enumerate}
\end{thm}
For $g\in G$, let $C_V (g)\leq V$ denote the fixed point space of
$g$. We can naturally define $C_V(H)=\cap_{h\in H}C_V(h)$ for each
$H\subset G$.  By choosing an $x\in V$ such that $C_G(x)$ is
isomorphic to the minimal stabilizer, we prove that if $H_1,
H_2,\dots, H_t$ are the minimal subgroups of $C_G(x)$ then
$\cup_iC_V(H_i)\neq V$. Consequently, for any $y\in
V\setminus\cup_iC_V(H_i)$ the pair $x,y$ is a base of size 2. To
complete this argument we need the following lemma.
\begin{lem}\label{lem:qscent}
  Let us choose $x\in V$ such that $C_G(x)$ is isomorphic to the minimal
  stabilizer. If one of the following holds, then there exists $y\in V$
  such that $C_G(x)\cap C_G(y) = 1$.
  \begin{enumerate}
  \item $C_G(x)$ has less than $p+1$ minimal subgroups.
  \item More generally, $C_G(x)$ has $r+t$ minimal subgroups
    $H_1,\ldots, H_{r+t}$ with $\dim C_V(H_i)\leq n-2$ for $r+1\leq i\leq r+t$ 
    and with $rp-r+t+1< p^2$.
  \end{enumerate} 
\end{lem}
\begin{proof}
  \begin{enumerate}
  \item Let $H_1, H_2,\dots, H_r$ be the minimal subgroups of $C_G(x)$ where
    $r<p+1$. Since the vector space $V$ over the field of $p$ elements
    cannot be covered with less than $p+1$ proper subspaces,
    we have $\cup_iC_V(H_i)\neq V$ so $C_G(x)\cap C_G(y)=1$ holds
    for any $y\in V\setminus \cup_iC_V(H_i)$.
  \item 
    Using the assumptions we have 
    \begin{multline*}
    \!\!\left|\cup_{i=1}^{r+t}C_V(H_i)\right|\leq
    rp^{n-1}-(r-1)p^{n-2}+tp^{n-2}
    =p^{n-2}(rp-r+t+1)< p^n.
    \end{multline*}
    Therefore, $\cup_{i=1}^{r+t}C_V(H_i)\neq V$ and $C_G(x)\cap C_G(y)=1$ holds
    for any element $y\in V\setminus \cup_iC_V(H_i)$.
  \end{enumerate}
  
\end{proof}
Now, we can prove the following
\begin{thm}\label{thm:quasisimple_main}
  Suppose $p$ is a prime, $V$ is an $n$-dimensional vector space over
  $\FF p$ and $G\leq GL(V)$ is a $p'$-group having a quasisimple
  normal subgroup $N$ which is irreducible on $V$. Then there exist
  $x,y\in V$ such that $C_G(x)\cap C_G(y)=1$.
\end {thm}
\begin{proof}
  We only have to deal with the cases where there is no regular
  orbit. By using Theorem \ref{thm:ghp}, first assume that $N =A_c$,
  the alternating group of degree $c$ for $c<p$ such that $V$ is the
  deleted permutation module for $N$. Let $W$ be the natural
  permutation module of $A_c$ over $\FF p$. Then $W\simeq V\oplus U$,
  where $U$ is the trivial module of $N$.

  Let $e_1, e_2, \dots, e_{c}\in W$ be a basis of $W$ permuted by $N$
  in the usual way and let $x\in V$ be the image of the vector $e_1 +
  2e_2 + 3e_3+\dots +ce_{c}$ by the projection onto $V$ along
  $U$. Obviously, $C_G(x)\cap N=1$. Let $K=C_G(x)\cap NZ(G)$.  Then
  $C_G(x)/K$ is isomorphic to a subgroup of $G/NZ(G)$.  We claim that
  $|G:NZ(G)|\leq 2$. As $V$ is an absolutely irreducible $\FF p
  N$-module, it follows that $C_G(N)=Z(G)$, hence $G/NZ(G)$ is
  embedded in $\out(A_c)$. This group is isomorphic to $C_2$, unless
  $c=6$.  For $c=6$ our claim follows from the observation that for
  any $g\in G\leq GL(V)$ the conjugation by $g$ on $N$ preserves the
  trace, so it fixes both conjugacy classes of elements of $N\simeq
  A_6$ of order $3$ (since they have trace $2$ and $-1$,
  respectively). However, every automorphism in $\aut(A_6)\setminus
  S_6$ moves one of these classes to the other, so $G/NZ(G)$ is
  embedded into $S_6/A_6\simeq C_2$, as claimed.  Thus, $C_G(x)/K\leq
  C_2$ holds for any $c$.

  On the other hand, using that $K\cap N=1$, it follows that $K$ is
  isomorphic to a subgroup of $Z(G)\leq \FF p^\times$ and $C_G(x)$
  acts trivially on $K$.  Hence $C_G(x)\geq K\geq 1$ is a central
  chain of $C_G(x)$ and $C_G(x)$ is Abelian. As any coprime Abelian
  linear group has a regular orbit, we get $G$ has a base of size two.

  Now, we investigate the groups listed in part \textit{2} of Theorem
  \ref{thm:ghp}.  Let us choose $x\in V$ such that $C_G(x)$ is
  isomorphic to the minimal stabilizer that can be seen in Table
  \ref{tab:regorb}.  Using part \textit{1} of Lemma \ref{lem:qscent}
  we can prove the existence of a suitable $y\in V$ unless $G,\ n$ and
  $p$ are one of the following

  \begin{center}
    \begin{tabular}{@{}lllll@{}}
      \toprule
      &$G$ &$n$ &$p$ & minimal stabilizer\\
      \midrule
      1&$U_4(2)\times Z$&$5$ &$7$ & $S_4$ \\
      2&$U_4(2).2\times Z$&$6$ &$7$ & $D_{12}$ \\
      3&$2.U_4(2)\star Z$&$4$ &$7$ & $U_{72}$ \\
      \hline
      4&$6_1.U_4(3).2_2\star Z$&$6$ &$13$ & $W(B_3)$ \\
      \hline
      5&$2.\Omega_8^+(2)\star Z$&$8$ &$11$ & $W(B_3)$ \\
      \bottomrule
    \end{tabular}
  \end{center} 

  We deal with these cases one by one.  If $C_G(x)$ is isomorphic to
  $S_4,\ W(B_3)$ or $D_{12}$, then we can use the complex character
  table of the group in order to find the dimension of the
  subspace $C_V(H)\leq V$ for a minimal subgroup $H\leq
  C_G(x)$. Indeed, in these cases every entry of the (complex)
  character table of $C_G(x)$ is a rational integer, so $C_G(x)$ has
  the same set of irreducible characters over $\FF p$ as it has over
  $\CC$ by using \cite[Theorem 15.13]{Ibook} and \cite[Theorem
  9.14]{Ibook}.

  \textbf{Case 1 ($\mathbf{C_G(x)\simeq S_4}$):} First, let us
  assume that $C_G(x)\simeq S_4$, so $V$ can be viewed as a faithful
  $\FF p S_4$-module.  Now, for a $\psi\in\irr(S_4)$ we have
  $A_4'\leq\ker\psi$ if and only if $\psi(1)\neq 3$.  It follows that
  $V$ contains a faithful irreducible $\FF p S_4$-submodule $U$ of
  dimension $3$.  Now, the character of $S_4$ related to $U$ is one of
  the following:
  \[
  \begin{array}{@{}|c|ccccc|@{}}
    \hline
    &(1)	&(12)	&(12)(34)	&(123)	&(1234)\\
    \hline
    \psi_1	&3  	&1	&-1		&0	&-1	\\
    \psi_2	&3  	&-1	&-1		&0	&1	\\
    \hline
  \end{array}
  \]
  Let $X_1$ and $X_2$ be the matrix representations of $S_4$ over
  $\overline{\FF p}$ with characters $\psi_1$ and $\psi_2$,
  respectively.  It follows easily from the character values and the
  order of the elements that
  \begin{gather*}
    X_1((12))\simeq \diag(1,1,-1),\quad X_2((12))\simeq
    \diag(1,-1,-1),\\ X_1((12)(34))\simeq
    X_2((12)(34))\simeq\diag(1,-1,-1),\\ X_1((123))\simeq X_2((123))\simeq
    \diag(1,\varepsilon,\varepsilon^2)
  \end{gather*}
  where $\varepsilon\in\overline{\FF p}$ is a primitive 3rd root of
  unity.  If $H$ is a minimal subgroup of $S_4$, then $H=\langle
  h\rangle$ for some $h\in S_4$ of order $2$ or $3$.  It follows from
  the above matrix forms that for an element $h\in C_G(x)$ of order
  $2$ or $3$ we have $\dim C_U(h)\leq 1$ for $h\in A_4$ and $\dim
  C_U(h)\leq 2$ for $h\in S_4\setminus A_4$.  Using part 2 of Lemma
  \ref{lem:qscent} with $r=6 ,\ t=7$ and $p=7$ we can choose a
  $y\in U\leq V$ such that $C_G(x)\cap C_G(y)=1$ holds.

  \textbf{Case 4 and 5 ($\mathbf{C_G(x)\simeq W(B_3)}$):} Now,
  $C_G(x)\simeq W(B_3)\simeq S_4\times K$ for $K=\langle
  k\rangle\simeq C_2$, so every $\chi\in\irr(W(B_3))$ is of the form
  $\chi=\psi\times\lambda$ for some $\psi\in\irr(S_4)$ and
  $\lambda\in\irr(K)$.  Again, $W(B_3)$ acts faithfully on $V$, so, by
  using the same argument as in Case 1, $V$ contains an irreducible
  $\FF p W(B_3)$-submodule $U$ of dimension $3$ and $S_4\leq W(B_3)$
  acts faithfully on $U$. If $K$ acts trivially on $U$, then, by Case
  1, we can find a $y_1\in U$ such that $C_G(x)\cap C_G(y_1)=K$. Let
  $V=U\oplus U_2$ for some $\FF pC_G(x)$-submodule $U_2$. As $K$ acts
  faithfully on $U_2$, there is a $y_2\in U_2$ such that
  $C_K(y_2)=1$. By choosing $y=y_1+y_2$ we get $C_G(x)\cap C_G(y)=1$.
  If $K$ acts faithfully on $U$, let $\chi=\psi\times\lambda$ be the
  character related to $U$. Then $\psi\in\irr(S_4)$,\ $\psi(1)=3$ and
  $\lambda(k)=-1$ holds.  By Case 1, either $\psi=\psi_1$ or
  $\psi=\psi_2$.  In either case, $U=C_U(g) \oplus C_U(gk)$ for any
  $g\in S_4$ with $o(g)=2$. Therefore,
  \begin{multline*}
    |\{H\leq C_G(x)\,|\,|H|=2,\,\dim C_U(H)=1\}|\\
    =|\{H\leq C_G(x)\,|\,|H|=2,\,\dim C_U(H)=2\}|=9.
  \end{multline*}
  Now, we can apply part 2 of Lemma \ref{lem:qscent} with $r=9,\ t=14$
  and $p\geq 11$ to get a $y\in U$ such that $C_G(x)\cap C_G(y)=1$.

  \textbf{Case 2 ($\mathbf{C_G(x)\simeq D_{12}}$):} 
  In this case we have $C_G(x)\simeq D_{12}= \langle f,t\,|
  \,f^6=t^2=1,\;tft=f^{-1}\rangle$.  Now, $D_{12}'=\langle f^2\rangle$
  and $D_{12}$ has four linear characters and 2 irreducible characters
  of degree two.  As $D_{12}$ acts faithfully on $V$ it follows that
  $V$ contains an irreducible $\FF p D_{12}$-submodule $U$ of
  dimension $2$.  Now, the character related to $U$ is one of the
  following
  \[
  \begin{array}{@{}|c|cccccc|@{}} 
    \hline
    &1&f^3	&\{f,f^5\}&\{f^2,f^4\}	&\{t,f^2t,f^4t\}&\{ft,f^3t,f^5t\}\\
    \hline
    \psi_1	&2&2	&-1  &-1	&0		&0		\\
    \psi_2	&2&-2	&1	  &-1	&0		&0		\\
    \hline
  \end{array}
  \] 
  If the character related to $U$ is $\psi_1$, then
  $D_{12}/\ker(\psi_1)\simeq D_6$ acts faithfully on $U$, so, by using
  part \textit{1} of Lemma \ref{lem:qscent} we can find a $y_1\in U$
  such that $C_G(x)\cap C_G(y_1)=\langle f^3\rangle$. Let $V=U\oplus
  U_2$ for some $\FF p C_G(x)$-submodule $U_2$.  Now, $\langle
  f^3\rangle$ acts faithfully on $U_2$, so $C_{\langle
    f^3\rangle}(y_2)=1$ for some $y_2\in U_2$. By choosing $y=y_1+y_2$
  we get $C_G(x)\cap C_G(y)=1$.  If the character related to $U$ is
  $\psi_2$, then $D_{12}$ acts faithfully on $U$. Now, $D_{12}$ has
  seven subgroups of order $2$ and only one of order $3$, the one
  generated by $f^2$.  However, in this case the related matrix
  $X_2(f^2)\simeq \diag(\varepsilon,\varepsilon^2)$ for some
  $\varepsilon\in\overline{\FF p}$, where $\varepsilon$ is a primitive
  3rd root of unity, so $f^2$ does not fix any vector in
  $U\setminus\{0\}$. Applying part 2 of Lemma \ref{lem:qscent} with
  $r=7,\ t=1$ and $p=7$ we get a $y\in U$ such that $C_G(x)\cap
  C_G(y)=1$.

  \textbf{Case 3 ($\mathbf{C_G(x)= U_{72}}$):} Although one can show
  that also in this case there is a regular orbit on $V$ for $C_G(x)$,
  we present here a simpler proof by showing that there is a
  two-dimensional subspace of $V\simeq \FF 7^4$ with trivial pointwise
  stabilizer in $G=2.U_4(2)\star \FF 7^\times\simeq 2.U_4(2)\times
  C_3$.  Let $N=2.U_4(2)$ and $a\in G$ a central element of order $3$,
  so $G=N\times \langle a\rangle$. For each prime divisor $p\mid |G|$,
  i.e. for $p=2,3,5$, let
  \[
  n_p:=|\{U\leq V\,|\,\dim U=2,\ \exists g\in G:o(g)=p\textrm{ and } U\leq 
  C_V(g)\}|.
  \]
  
  By using the character table of $N=2.U_4(2)$ found in Atlas
  \cite[page 27]{atlas}, one sees that the corresponding Brauer
  character of the $\FF 7 G$-module $V$ is one of the two conjugate
  characters $\chi_{21},\chi_{22}\in\irr(N)$. Since we are only
  interested in dimensions of fixed point spaces of elements of
  $G$, we can assume that the Brauer character of $V$ is
  $\chi_{21}$.

  Now, for each prime divisor $p\mid |G|$ and for
  each conjugacy class $\mathcal C\subseteq N$ consisting of elements
  of order $p$, the following table contains the size of $\mathcal C$,
  the character value of $\chi_{21}$ on $\mathcal C$, and its unique
  decomposition to the sum of four $p$-th roots of unity.
  (From this point of the proof $\ve$ and $\omega$ denote 
  a primitive third and fifth root of unity, respectively.)
  
  \begin{center}
  \begin{tabular}{@{}ccccc@{}}
    \toprule
    p&conj.\ class ($\mathcal C$)& $|\mathcal C|$ & $\chi_{21}(g)\ 
    (g\in\mathcal C)$&decomposition\\
    \midrule
    2&2A&1  &-4&-1-1-1-1\\
     &2B&90 &0 &-1-1+1+1\\
    \midrule
    3&3A& 40&$-1/2 + 3\sqrt{3}/2i$&$1+1+1+\ve$\\
     &3B& 40&$-1/2 - 3\sqrt{3}/2i$&$1+1+1+\ve^2$\\
     &3C&240&-2&$\ve+\ve+\ve^2+\ve^2$\\
     &3D&480&1 &$1+1+\ve+\ve^2$\\
    \midrule
    5&5A&5184&-1&$\omega+\omega^2+\omega^3+\omega^4$\\
    \bottomrule
  \end{tabular}
  \end{center}

  It follows that $n_2\leq 90$ and $n_5=0$. Furthermore, 
  if $(n,a^k)\in N\times\langle a\rangle$ is an element of 
  order $3$, then 
  \[
  \dim C_V((n,a^k))=\left\{
    \begin{array}{l@{\quad}l}
      3&\textrm{ if }n\in 3A\cup 3B,\ a^k=1,\\
      2&\textrm{ if }n\in 3C,\ a^k\neq 1,\\
      2&\textrm{ if }n\in 3D,\ a^k=1,\\
      \leq 1&\textrm{ otherwise}.
    \end{array}
  \right.
  \]
  As the number of two-dimensional subspaces in a three dimensional
  subspace of $V$ is $57$ we get
  \begin{multline*}
  \left|\left\{(U,g)\,|\,U\leq V,\ \dim U=2,
    \ g\in G,\ o(g)=3,\ U\leq C_V(g)\right\}\right|\\
  \leq 2\cdot 40\cdot 57+240\cdot 2+480=5520.
  \end{multline*}
  Since $C_V(g)=C_V(g^{-1})$ for any $g\in G$ and the intersection of
  two three dimensional subspaces in $V$ is a two dimensional subspace,
  we get $n_3\lneqq \frac{5520}{2}=2760$.

  Hence the number of two dimensional subspace of $V$ fixed pointwise
  by some minimal subgroup of $G$ is less than $90+2760=2850$. However, the
  number of all two dimensional subspaces of $V$ is exactly $2850$,
  which completes our proof.
\end{proof}
\section{Groups of symplectic type}\label{sec:symplectic}
The purpose of this section is to prove Theorem \ref{thm:main_2base}
for coprime linear groups of symplectic type. First, we define what do
we mean under such a group.
\begin{defin}\label{def:symp}
  Let $V$ be a vector space and $G\leq GL(V)$ a linear group.  We say
  that $G$ is of symplectic type if it has normal subgroups $Z\leq
  N\leq G$ such that $Z=Z(G)$, the quotient group $N/Z\nor G/Z$ is an
  elementary abelian $r$-subgroup for some prime number $r$ and $V$ is
  an absolutely irreducible $KN$-module.
\end{defin}
Throughout this section let $V$ be an $n>1$ dimensional vector space 
over $K\simeq \FF q$ for some prime power $q$
and let $G\leq GL(V)$ be a coprime linear 
group of quasisimple type with normal subgroups $Z\leq N\leq G$ 
according to Definition \ref{def:symp}. We also assume that 
$Z=Z(G)\simeq K^\times$ is the group of all scalar transformations.

With these assumptions $N=R\star Z$, 
where $R$ is an extraspecial $r$-group, $R\cap Z\simeq C_r$ and 
$V$ is a faithful and absolutely irreducible representation of $R$ over $K$. 
\begin{defin}\label{def:mon}
  We say that $N$ is not monomial, if $r=2$, $q\equiv -1\mod 4$, and
  $R$ is a central product of some dihedral groups $D_8$ of order
  eight by one quaternion group $Q$. Otherwise, we say that $N$ is
  monomial.
\end{defin}
\begin{rem}
  The explanation of our terminology is that $N$ is monomial if and
  only if $V$ is a monomial representation of $N$.  This will be shown
  later as part of a more detailed description of the $KN$-module $V$.
\end{rem}
\subsection{Finding a two-element base in case when 
  \texorpdfstring{$N$}{N} is monomial}\label{subsec:mon}
In the following we assume that $N$ is monomial.
\begin{thm}\label{thm:N_mon}
  There exist a decomposition $N=D\rtimes S$ and a suitable basis
  $\{\mb{n}\}\sbs V$ with the following properties:
  \begin{enumerate}
  \item $D=Z\times D_1$ and $D_1\simeq S\simeq C_r^k$ for $r^k=n$.
  \item With respect to $\mb{n}$, the subgroup $S\leq GL(n,K)$
    consists of permutation matrices. Moreover, $S$ acts
    regularly on this basis.
  \item With respect to $\mb n$, the subgroup $D\leq GL(n,K)$ consists
    of diagonal matrices.  The subspaces $\langle u_i\rangle,\ 1\leq
    i\leq n$ are all the irreducible representations of $D_1$ with
    $\langle u_1\rangle$ being the trivial representation of $D_1$,
    and they are pairwise non-equivalent. Moreover, the main diagonal
    of every $g\in D_1\leq GL(n,K)$ contains all of the $o(g)$-th
    roots of unity with the same multiplicity.
  \end{enumerate}
\end{thm}
\begin{proof}
  It is well-known that any extraspecial group of order $r^{2k+1}$ is
  a product of two Abelian subgroups of order $r^{k+1}$.  First we
  prove that in our case we can assume that $R$ is a product of two
  elementary abelian $r$-groups.  For $r>2$ the elements of orders
  dividing $r$ form a characteristic subgroup in $R$ (since in this
  case $(xy)^r=x^ry^r$ holds in $R$).  The product of this subgroup
  with $Z$ is a normal subgroup of $G$.  As there is no proper normal
  subgroup of $G$ between $Z$ and $N$, we get $\exp(R)=r$ if $r>2$.
  If $R=R_1\star R_2\star\ldots\star R_k$ is a central product
  of $k$ dihedral groups $D_8$ of order $8$, we can choose elements
  $x_i,y_i\in R_i$ of order $2$ such that $R_i=\langle
  x_i,y_i\rangle$. In this case $\langle x_1,\ldots,x_k,Z(R)\rangle$
  and $\langle y_1,\ldots,y_k,Z(R)\rangle$ are two elementary abelian
  subgroups whose product is $R$.  If $R$ is a central product
  of $k-1$ dihedral groups $D_8$ by a quaternion group $Q=\langle
  i,j\rangle\leq R$ and $q\equiv 1\mod 4$, let $\lambda\in Z$ be an
  element of order 4. By defining $H=\langle\lambda i,\lambda
  j\rangle\leq Z Q$ we get $H\simeq D_8$ and $Z H=Z Q$.  So in this
  case we can replace $R$ by a central product of $k$ dihedral groups,
  and we can apply the previous argument.

  Let $R=D_R S_R$ be a product of two elementary abelian groups of
  order $r^{k+1}$, $D=Z D_R$ and $S$ a complement of $Z(R)$ in
  $S_R$. Then $N=D\rtimes S$.  As $K$ contains the $\exp(D)$-th roots
  of unity, every irreducible $K$-representation of $D$ is
  one-dimensional.  Fix an $ u_1\in V$ such that $K u_1$ is a
  one dimensional $D$-invariant subspace. Choosing $D_1=C_D( u_1)$
  we have $D=Z\times D_1$ and $D_1\simeq S\simeq C_r^k$. It is
  well-known that any non-linear absolutely irreducible representation
  of an extraspecial group of order $r^{2k+1}$ has degree $r^k$, hence
  $r^k=n$ and \textit{1} holds.

  Let $S=\{s_1,s_2,\ldots,s_n\}$ and for each $1\leq i\leq n$ let $
  u_i$ be defined as $ u_i=s_i( u_1)$.  Since $D\nor N$, by
  Clifford theory we have $K u_i$ is a $D$-invariant subspace for
  each $1\leq i\leq n$.  Hence $\langle\mb n\rangle$ is a
  $DS=N$-invariant subspace, so it is equal to $V$. As $n=|S|=\dim V$,
  it follows that $\mb n$ is a basis of $V$ and $S$ acts regularly on
  this basis.  So \textit{2} follows.

  We have already seen in the last paragraph that $\langle
  u_i\rangle$ is a $D$-invariant subspace for each $i$, hence $D\leq
  GL(n,K)$ consists of diagonal matrices with respect to $\mb n$.  The
  claim that $\langle  u_i\rangle$ are pairwise non-equivalent
  $KD_1$-modules follows easily from the fact $C_S(D_1)=1$.  Indeed,
  let $ u_i\neq  u_j$ be two basis elements. Then $ u_j=s
  u_i$ for some $s\in S\setminus \{1\}$. Let $d\in D_1$ such that
  $[d,s]\in Z\setminus\{1\}$, and let $d=\diag(d_1,\ldots,d_n)$ be the
  diagonal form of $d$. Then $d_j u_j=[d,s](d_i u_j)$, so
  $d_j\neq d_i$ which proves that $\langle  u_i\rangle$ and
  $\langle  u_j\rangle$ are non-isomorphic $KD_1$-modules.  As
  $|D_1|=n$, these are all the irreducible representations of
  $D_1$. Furthermore, $\langle u_1\rangle$ is the trivial
  representation of $D_1$ by definition.  Finally, if $g\in D_1$, then
  any linear representation of $\langle g\rangle$ can be extended to
  $D_1$ in exactly $|D_1|/o(g)$ ways, which completes the proof of
  \textit{3}.
\end{proof}
\begin{thm}\label{thm:u1_centr}
  Let $g\in G$ be any group element fixing $ u_1$. Then
  \begin{enumerate}
  \item $D_1^g=D_1$ and $g$ is a monomial matrix.  Hence there exists
    a unique decomposition $g=\delta(g)\pi(g)$ to a product of a
    diagonal matrix $\delta(g)$ and a permutation matrix $\pi(g)$.
  \item $\pi(g)$ normalizes $S$, that is, $S^{\pi(g)}=S$.
  \item Both $\delta(g)$ and $\pi(g)$ normalize $N$, so
    $N=N^{\delta(g)}=N^{\pi(g)}$. Moreover, $[\delta(g), S]\sbs D$.
  \item If $\delta(g) \neq 1$, then the number of $1$'s
    in the main diagonal of $\delta(g)$ is at most $\frac{3}{4}n$.
  \end{enumerate}
\end{thm}
\begin{proof}
  The statement $D_1^g=D_1$ follows from the fact $D_1=C_N(
  u_1)\nor C_G( u_1)$.  Consequently, $g$ permutes the homogeneous
  components of the $D_1$-module $V$. By part \textit{3} of Theorem
  \ref{thm:N_mon}, these homogeneous components are just the
  one-dimensional subspaces $\left< u_i\right>$ for $1\leq i\leq
  n$. It follows that $g$ is a monomial matrix.  Of course, a monomial
  matrix $g$ has a unique decomposition $g=\delta(g)\pi(g)$, and part
  \textit{1} is proved.

  The map $\pi:g\rightarrow \pi(g)$ defines a homomorphism from the
  group of monomial matrices into the group of permutation
  matrices. As $g\in G$ normalizes $N$, we have $\pi(g)$ normalizes
  $\pi(N)=S$ and \textit{2} follows.

  Both $g$ and $\delta(g)$ normalize $D$, hence
  $\pi(g)=\delta(g)^{-1}g$ also normalizes $D$.  We have already seen
  that $\pi(g)$ normalizes $S$, so it also normalizes
  $N=DS$. Therefore $\delta(g)=g\pi(g)^{-1}$ also normalizes
  $N$. Finally, $[\delta(g), S]$ is a subset of $N$ and it consists of
  diagonal matrices, so $[\delta(g), S]\sbs D$ and \textit{3} holds.

  If $\delta(g)\neq 1$, then $\delta(g)$ is not a scalar matrix, so
  there exists an $s\in S$ such that $[\delta(g), s] \in D\setminus \{
  1\}$. Using part \textit{3} of Theorem \ref{thm:N_mon}, we get that the
  number of $1$'s in the main diagonal of $[\delta(g), s]$ is at most
  $\frac{1}{2}n$.  This cannot be true if the number of $1$'s in
  $\delta(g)$ is more than $\frac{3}{4}n$. We are done.
\end{proof}
By part \textit{2} of Theorem \ref{thm:N_mon}, $S$ acts regularly on
the basis $W=\{ u_1,\ldots, u_n\}$. Identifying $ u_1$ with
the unit element of $S$, this action defines a vector space structure
on $W$ isomorphic to $\FF r^k$. Viewing $W$ in this way as an $\FF
r$-vector space, the zero element of $W$ is $ u_1$.  (Here it may be a
bit confusing that $W$ is a basis of the original space, while it
itself has a vector space structure inherited from the regular action
of $S$ on $W$). Using the previous theorem, $C_G( u_1)$ consists of
monomial matrices and its permutation part $\pi(C_G( u_1))$ acts by
conjugation on $S$.  In fact, this action is faithful, since
$C_{GL(n,K)}( u_1)\cap C_{GL(n,K)}(S)=1$. It follows that the action of 
$\pi(C_G( u_1))$ on $W$ is linear, when we consider $W$ as a vector
space, so $\pi(C_G( u_1))\leq GL(W)\simeq GL(k,r)$.  Choosing a basis
$\{e_1, e_2,\ldots, e_k\}\sbs W$, the next theorem helps us to find a
``good'' $\pi(C_G( u_1))$-regular partition of $W$.
\begin{thm}\label{thm:reg_partition}
  Let $H\leq GL(W)$, where $r$ is a prime and $W\simeq \FF r^k$ is an
  $\FF r$-space with basis $ e_1, e_2,\ldots, e_k$. Then there is an
  $H$-regular partition $W=\cup_i\Omega_i$ of $W$ with the following
  properties
  \begin{enumerate}
  \item For $r\geq 3,\ k=1$ we have $W=\{ e_1\}\cup \Omega_2$ with
    $|\Omega_1|<\frac{1}{4}|W|$ if $|W|\neq 3$.
  \item For $r\geq 3,\ k\geq 2$ we have $W=\{ e_1\}\cup \Omega_2
    \cup \Omega_3$ such that $|\Omega_2|+2<\frac{1}{4}|W|$ if $|W|\neq
    9$.
  \item For $r=2,\ k\geq 3$ we have $|W|=\{ e_1\} \cup \{ e_2\}
    \cup \Omega_3 \cup \Omega_4$ such that $|\Omega_3|<\frac{1}{4}|W|$ 
    if $|W|\neq 16$.
  \end{enumerate}
\end{thm}
\begin{proof}
  First, let us assume that $r\geq 3$ and $k=1$. Then $\Omega_1=\{
  e_1\}$ is a basis of the one dimensional space $W$ with
  $|\Omega_1|=1<\frac{5}4\leq \frac{1}4|W|$ if $r\geq 5$.

  In case of $r\geq 3,\ k\geq 2$ let 
  \begin{gather*}
    \Omega_1=\{ e_1\},\quad
    \Omega_2=\{ e_2, e_3,\ldots, e_k,
     e_1+ e_2, e_2+ e_3,\ldots, e_{k-1}+ e_k\},\\
    \Omega_3=W\setminus(\Omega_1\cup\Omega_2).
  \end{gather*}
  To prove that this is an $H$-regular partition of $W$, let $h\in
  H_{\Omega_1}\cap H_{\Omega_2}$.  By using induction on $t$, we prove
  that $h e_t= e_t$ for each $1\leq t\leq k$, that is, $h=1$.
  First, our choice $\Omega_1=\{ e_1\}$ guarantees that $h
  e_1= e_1$.  Assuming that $h e_i= e_i$ for all $1\leq
  i<t\leq k$, it follows that $h( e_t)$ and $h( e_{t-1}+
  e_t)$ are elements of the set
  \[
  \Omega_2\setminus\langle e_1,\ldots, e_{t-1}\rangle=\{
  e_t, e_{t+1},\ldots, e_k, e_{t-1}+ e_t,\ldots,
  e_{k-1}+ e_{k}\}.
  \]
  Since $h(e_{t-1}+ e_t)-h( e_t)=e_{t-1}$, it follows that either
  $h(e_t)$ or $h(e_{t-1}+ e_t)$ contains $e_{t-1}$ with non-zero
  coefficient. However, there is only one element in
  $\Omega_2\setminus\langle e_1,\ldots, e_{t-1}\rangle$ with this
  property, namely, $e_{t-1}+ e_t$.  So either $h( e_{t-1}+ e_t)=
  e_{t-1}+ e_t$ or $h( e_t)= e_{t-1}+ e_t$. In the latter case
  $h(e_{t-1}+ e_t)=2 e_{t-1}+ e_t\not\in \Omega_2$, since $r\neq 2$, a
  contradiction. It follows that $h( e_{t-1}+ e_t)= e_{t-1}+ e_t$, so
  $h(e_t)=h(e_{t-1}+ e_t)-h( e_{t-1})= e_t$. Thus, $h e_t= e_t$ for
  each $1\leq t\leq k$, which proves that the given partition is
  $H$-regular.  The inequality
  $|\Omega_2|+2=2k<\frac{1}{4}r^k=\frac{1}{4}|W|$ holds for $k\geq 2,\
  r\geq 5$ or for $k\geq 3,\ r=3$, but it fails for $k=2,\ r=3$, which
  proves \textit{2}.

  For $r=2,\ k=3$ let
  \begin{gather*}
    \Omega_1=\{ e_1\},\quad \Omega_2=\{
    e_2\},\quad \Omega_3=\{ e_3\},\\
    \Omega_4=W\setminus(\Omega_1\cup\Omega_2\cup\Omega_3),
  \end{gather*}
  while for $r=2,\ k\geq 4$ let
  \begin{gather*}
    \Omega_1=\{ e_1\},\quad \Omega_2=\{ e_2\},\\
    \Omega_3=\{ e_3,\ldots, e_k,  e_2+ e_3,
     e_3+ e_4, e_4+ e_5,\ldots, e_{k-1}+ e_k,
     e_k+ e_1\},\\
    \Omega_4=W\setminus(\Omega_1\cup\Omega_2\cup\Omega_3).
  \end{gather*}
  Now, for $k=3$, the above partition is clearly $H$-regular, since we
  fixed each element of the basis $ e_1, e_2,
  e_3$. Furthermore, $|\Omega_3|=1<2=\frac{1}{4}|W|$ holds. In case
  $k\geq 4$, we prove that the given partition is $H$-regular by using
  a similar argument as we did in case \textit{2}. Let $h\in H$ be an
  element fixing every element of the partition.  Assuming that $h(
  e_i)= e_i$ for all $1\leq i<t<k$ we get that $h( e_t)$ and $h(
  e_{t-1}+ e_t)$ are elements of the set
  \[
  \Omega_3\setminus\left< e_1,\ldots, e_{t-1}\right>=\{
  e_t, e_{t+1},\ldots, e_k, e_{t-1}+ e_t,\ldots,
  e_{k-1}+ e_{k}, e_k+ e_1\}.
  \]
  Since $h( e_t)+h( e_{t-1}+ e_t)= e_{t-1}$, we have
  either $h( e_{t-1}+ e_t)= e_{t-1}+ e_t$ or $h(
  e_t)= e_{t-1}+ e_t$.  In the former case we get $h(
  e_t)= e_t$, while in the latter case we can take $ e_t+
  e_{t+1}\in\Omega_3$ since $t<k$. Now $ e_{t-1}$ occurs with $0$
  coefficient both in $h( e_t+ e_{t+1})$ and $h( e_{t+1})$,
  since the only element of $\Omega_3\setminus\left< e_1,\ldots,
    e_{t-1}\right>$ containing $ e_{t-1}$ with nonzero coefficient
  is $h( e_t)$.  However, $h( e_t+ e_{t+1})+h(
  e_{t+1})=h( e_t)= e_{t-1}+ e_t$, a contradiction.  It
  remains to prove that $h( e_k)= e_k$.  It is clear that
  \[
  h( e_k)\in\Omega_3\setminus\left< e_1, e_2,\ldots,
    e_{k-1}\right>= \{ e_k, e_{k-1}+ e_k,  e_k+ e_1\}.
  \]
  If $h( e_k)= e_{k-1}+ e_k$, then $h( e_k+ e_1)=
  e_{k-1}+ e_k+ e_1\not\in\Omega_3$.  If $h( e_k)=
  e_{k}+ e_1$, then $h( e_{k-1}+ e_k)= e_{k-1}+
  e_{k}+ e_1\not\in\Omega_3$.  Thus $h( e_k)= e_k$ also
  holds. It follows that $h=1$, so the given partition is $H$-regular.
  Finally, $|\Omega_3|=2k-3<2^{k-2}=\frac{1}{4}|W|$ if $|W|>16$.
\end{proof}
\begin{thm}\label{thm:rneq2}
  With the above notation, let us assume that $r\geq 3$ and
  let $x$ be defined as $x= u_1$. Furthermore, let $\alpha\in K$ be
  a generator of the multiplicative group of $K$. By using parts
  \textit{1-2} of Theorem \ref{thm:reg_partition} let $y$ be defined
  as follows
  \begin{align*}
    &\textrm{For\ \ } k=1,\ |W|\neq 3:&&y=\sum_{u_i\in\Omega_2} u_i,\\
    &\textrm{For\ \ } k\geq 2,\ |W|\neq 9:&&y=
    \alpha e_1+\sum_{u_i\in\Omega_3} u_i,\\
    &\textrm{For\ \ } |W|=9:&& y=\alpha e_1+0\cdot e_2+
    \!\!\sum_{u_i\notin\{e_1, e_2\}}\!\! u_i,\\
    &\textrm{For\ \ } |W|=3:&&
    y=\alpha u_2+ u_3.
  \end{align*}
  Then $C_G(x)\cap C_G(y)=1$.
\end{thm}
\begin{proof}
  Let $g\in C_G(x)\cap C_G(y)$. Since $g\in G$ fixes $x=u_1$, we get
  that $g$ is a monomial matrix, thus we have a decomposition $g=\dg
  \pg$.  Notice also that if the monomial matrix $g$ fixes $y$, then
  $\pi(g)$ permutes those basis elements of $W$ among each other which
  appear in $y$ with zero coefficients.

  In the first case $\pg$ fixes $\Omega_1$.  As
  $W=\Omega_1\cup\Omega_2$ is a $\pi(C_G( u_1))$-regular partition, we
  get $\pi(g)=1$. Hence $g=\dg$ is a diagonal matrix.  If $g_{i}$
  denotes the $i$-th element of the main diagonal of $g$, then $g\in
  C_G(y)$ holds only if $g_i=1$ for all $u_i\in\Omega_2$. As
  $|\Omega_2|>\frac{3}{4}|W|$, by using part \textit{4} of Theorem
  \ref{thm:u1_centr}, we get $g=\dg=1$.

  In the second case we see that $\pg$ fixes the subset $\Omega_2\sbs
  W$, since only the elements of $\Omega_2$ occur with coefficient $0$
  in $y$. However, in this case it is possible that $\pi(g)$ moves the
  unique element of $\Omega_1$ into an element of $\Omega_3$. Of
  course, in that case it moves an element of $\Omega_3$ into the
  element of $\Omega_1$. This results in the appearance of an $\alpha$
  and an $\alpha^{-1}$ in the main diagonal of $\dg$. It follows that
  in the main diagonal of $\dg$ the number of elements different from
  $1$ is at most $|\Omega_2|+2$, which is less than $\frac{1}{4}|W|$
  by Theorem \ref{thm:reg_partition}. Using part of \textit{4} of
  Theorem \ref{thm:u1_centr} we get $\dg=1$, hence $\pi(g)$ also fixes
  the unique element of $\Omega_1$, so $g=\pi(g)=1$.

  In case $|W|=9$ the matrix $\pg$ fixes $e_2$.  If $\pi(g)$ does
  not fix $e_1$, then in the main diagonal of $\dg$ there is an
  $\alpha$ and an $\alpha^{-1}$ and at most one more element not
  equal to $1$. Since $S$ acts regularly on $W$, we can choose an
  element $s\in S$ which takes the basis element corresponding to
  $\alpha^{-1}$ into the basis element corresponding to
  $\alpha$. Then, the main diagonal of $[\dg,s]\in D$ contains an
  $\alpha^2\neq 1$ and at least four $1$'s.  However, there is no such
  element in $D$ by part \textit{3} of Theorem \ref{thm:N_mon}, a
  contradiction.  Hence $\pi(g)$ also fixes $e_1$, so $\pi(g)=1$.
  Furthermore, there can be at most one element in the main diagonal of
  $g=\delta(g)$ which is different from $1$, namely, that $g_i$,
  for which $u_i=e_2$. Using part \textit{4} of Theorem
  \ref{thm:u1_centr} we get $g=1$.

  Finally, let $|W|=3$. First, let $\alpha^3\neq 1$. If $g$ would not be a
  diagonal matrix, then
  \[
  \dg=
  {\s
    \begin{pmatrix}1&0&0\\0&\alpha&0\\0&0&\alpha^{-1}\end{pmatrix}
  }
  \textrm{\ and\ }
  [\dg,s]=
  {\s
    \begin{pmatrix}\alpha&0&0\\0&\alpha^{-2}&0\\0&0&\alpha\end{pmatrix}
  },
  \textrm{\ for\ }
  s=
  {\s
    \begin{pmatrix}0&0&1\\1&0&0\\0&1&0\end{pmatrix}
  }\in S.
  \]
  Since $\alpha\neq \alpha^{-2}$, we have $[\dg,s]\notin D$ by part
  \textit{3} of Theorem \ref{thm:N_mon}, which contradicts part
  \textit{3} of Theorem \ref{thm:u1_centr}. So $g$ is a diagonal matrix.

  If $\alpha^3=1$, that is, $q=4$, then $\pi(C_G(u_1))=1$ by using the
  assumption $(|G|,|V|)=1$, so $g$ is again a diagonal matrix.  

  Since each basis element appears either in $x$ or in $y$,
  $C_G(x)\cap C_G(y)$ cannot contain any diagonal matrix different
  from $1$, which completes our proof.
\end{proof}
Still assuming that $N$ is monomial, now we handle the case $r=2$, that
is, $n=2^k$ for some $k$.  In case $n=2$ any basis will
obviously be good; let for example $x= u_1, y= u_2$.  Now, we analyze
the case $n=4$. In accordance with Theorem \ref{thm:N_mon}, we choose a
basis $\{u_1, u_2, u_3, u_4\}\sbs V$.  In this case $N=ZD_1S$, where the
Klein groups $D_1=\left<d_1, d_2\right>$ and $S=\left<s_1,s_2\right>$
are generated (independently from the base field) by the matrices:
\begin{gather*}
  d_1\!=\!
  {\s\begin{pmatrix}1&0&0&0\\  0& 1& 0& 0\\  
      0& 0&-1& 0\\  0& 0& 0 &-1\end{pmatrix}},\quad
  d_2\!=\!
  {\s\begin{pmatrix}   1&0&0&0\\  0&-1& 0& 0\\
      0& 0& 1& 0\\  0& 0& 0 &-1 \end{pmatrix}},\\[6pt]
  s_1\!=\!
  {\s\begin{pmatrix}   0&0&1&0\\  0& 0& 0& 1\\ 
      1& 0& 0& 0\\  0& 1& 0 & 0 \end{pmatrix}},\quad
  s_2\!=\!
  {\s\begin{pmatrix}   0&1&0&0\\  1& 0& 0& 0\\ 
      0& 0& 0& 1\\  0& 0& 1 & 0   \end{pmatrix}}.
\end{gather*}
If the size of the base field is not equal to $3,5$ or $9$, then the
following theorem guarantees the existence of a two-element base
$\{x_0,y_0\}\in V$. (For a more consistent notation, in the next two
theorems the elements of the base are denoted by $x_0,y_0$ instead of
$x,y$ because they will be used in the constructions given for the case
$n=2^k,\ k\geq 3$.)
\begin{thm}\label{thm:GL47}
  Let $N=Z\left<d_1,d_2,s_1,s_2\right> \nor G\leq GL(4,q)$, and assume
  that $q\neq 3,\,5,\,9$.  Furthermore, let $\alpha\in \FF q$ such
  that $\alpha^8\neq 1$. Set $x_0=u_1$ and $y_0= u_2+\alpha
  u_3+\alpha^{-1} u_4$.  Then $C_G(x_0)\cap C_G(y_0)=1$.
\end{thm}
\begin{proof}
  Let $g\in C_G(x_0)\cap C_G(y_0)$.  By the choice of $x_0$ we know that $g$
  is a monomial matrix.  The first element in the main diagonal of
  $\delta(g)$ is 1, and the others are from the set
  $\{1,\alpha,\alpha^{-1},\alpha^2,\alpha^{-2}\}$. If $\delta(g)$
  contains an $\alpha$ or an $\alpha^{-1}$, then for some $s\in S$ we
  get $[\delta(g),s]\in D=Z\times D_1$ contains both $\alpha$ and
  $\alpha^{-1}$. By part \textit{3} of Theorem \ref{thm:N_mon}, this
  is impossible unless $o(\alpha^2)|4$ which does not hold.  It
  follows that either $g=1$, or
  \[
  g=
  \begin{pmatrix}     1\,&0&0&0\\ 0\,&1          &0&0\\ 
    0\,&0&0&\alpha^2\\ 0\,&0&\alpha^{-2}&0\end{pmatrix},\qquad\textrm{and}\qquad
  [\delta(g),s_1]=
  \begin{pmatrix}\alpha^2&0&0&0\\ 0  &\alpha^{-2}&0&0\\ 
    0&0&\alpha^{-2}&0\\0  &0&0&\alpha^2    \end{pmatrix}.
  \]
  In the latter case $[\delta(g),s_1]\in D$, so we get $o(\alpha^4)|2$, which
  is again impossible.
\end{proof}
In the remaining cases we have found the following pair of vectors 
by using the GAP system \cite{gap}.
\begin{thm}\label{thm:GL43}
  \begin{enumerate}
  \item 
    Let $N=Z\left<d_1,d_2,s_1,s_2\right>\leq GL(4,3)$ and let $G$ be the
    normalizer of $N$ in $GL(4,3)$.  Let $x_0,y_0\in V$ be 
    \[
    x_0=u_1,\qquad y_0= u_1+ u_3+ u_4.
    \]
    Then for a suitable $3'$-Hall subgroup $H\leq G$ we have
    $C_H(x_0)\cap C_H(y_0)=1$.  
  \item Let $N=Z\left<d_1,d_2,s_1,s_2\right>\leq GL(4,9)$ and let $G$
    be the normalizer of $N$ in $GL(4,9)$. Let $\alpha$ be a generator
    of the multiplicative group of $\FF 9$ and $x_0,y_0\in V$ be
    \[
    x_0=u_1+u_3+u_4,\qquad y_0=u_2+u_3+\alpha u_4.
    \]
    Then $C_G(x_0)\cap C_G(y_0)=1$.
  \item Let $N=Z\left<d_1,d_2,s_1,s_2\right>\leq GL(4,5)$ and let $G$
    be the normalizer of $N$ in $GL(4,5)$. Let $x_0,y_0\in V$ be
    \[
    x_0= u_1+ u_2+2 u_3,\qquad y_0= u_2+ u_3+2 u_4.
    \]
    Then $C_G(x_0)\cap C_G(y_0)=1$.
\end{enumerate}
\end{thm}
\begin{rem}
  In case one, $G$ does not have a two-element base. If $G_1\leq G$ is
  any $3'$-subgroup, then $g^{-1}G_1g\leq H$ for some $g\in G$. By
  applying $g$ to the basis elements we get $g(x_0),\ g(y_0)$ is a base for
  $G_1$.
\end{rem}
The constructions given in the last two theorems have the common
property that $y_0$ is a sum of exactly three basis vectors with
non-zero coefficient. Capitalizing this property, we shall give
a uniform construction in any case when
$N\nor G\leq GL(2^k,q)$ for all $k\geq 3$.

We note that in case of $n\geq 128$ 
we could give similar constructions as we did in Theorem \ref{thm:rneq2}.
However, for a more uniform discussion we alter these constructions a
bit, so it will be adequate even in smaller dimensions. The point of
our modification is that we do not choose $x$ as a basis element this time,
rather as a linear combination of exactly three basis vectors.
Although this effects that $C_G(x)$ will not be monomial any more,
we can cure this problem by an appropriate choice of $y$.

Again, let $\{ e_1, e_2,\ldots, e_k\}\sbs W=\{ u_1, u_2,\ldots, u_n\}$
be a basis of $W$ considering $W$ as a $\FF 2$-vector space (see the
paragraph before Theorem \ref{thm:reg_partition}).  Choosing the
indexing of the basis vectors $ u_1, u_2,\ldots, u_n$ appropriately,
we can assume that $\langle e_1, e_2\rangle_W=\{ u_1, u_2, u_3,
u_4\}$.  Let $V'=\left< u_1, u_2, u_3, u_4\right>\leq V$ be the
subspace generated by the first four basis vectors, and let $N_N(V')$
be the subgroup of elements of $N$ fixing $V'$.  The restriction of
$N_N(V')$ to $V'$ defines an inclusion $N_N(V')/C_N(V')$ into
$GL(V')$, so we get a subgroup $N'=Z\left<d_1,d_2,s_1,s_2\right>\leq
GL(V')$.  If $g\in N_G(V')$, then it is clear that $g_{V'}$ normalizes
$N'$. Using the results of Theorem \ref{thm:GL47} and Theorem
\ref{thm:GL43}, we can define $x_0,y_0\in V'$ such that $y_0$ is
the linear combination of exactly three basis vectors and $N_G(V')\cap
C_G(x_0)\cap C_G(y_0)$ acts trivially on $V'$.  Starting from the
vectors $x_0,y_0$, we search for a base $\{x,y\}\sbs V$ of the form
$x=y_0,\ y=x_0+v$, where $v\in V'':=\left< u_5, u_6,\ldots,
  u_n\right>$.  The following lemma indicates why this form is useful.
\begin{lem} \label{lem:x_three}
  $C_G(y_0)$ fixes both subspaces $V'$ and $V''$.
  As a result, for any $v\in V''$ we have that
  $C_G(y_0)\cap C_G(x_0+v)=C_G(y_0)\cap C_G(x_0)\cap C_G(v)$ acts
  trivially on $V'$. In particular, $C_G(y_0)\cap C_G(x_0+v)$ consists of
  monomial matrices.
\end{lem}
\begin{proof}
  It is enough to prove the inclusion $C_G(y_0)\leq N_G(V')\cap
  N_G(V'')$, the rest of the statement follows evidently.  Our proof
  is similar to the way we have proved that $C_G( u_1)$ consists of
  monomial matrices. As there are three basis elements in $y_0$ with
  non-zero coefficients and $S\simeq Z_2^k$ permutes the
  basis elements regularly, we get $C_N(y_0)\leq D$, i.e., every element of
  $C_N(y_0)$ is diagonal.  Hence every element of $C_N(y_0)$ fixes the
  three basis elements appearing in $y_0$. Using the assumption that
  $ u_1, u_2, u_3, u_4$ corresponds to a (two dimensional)
  subspace $S_2\leq S$, it follows easily that any element of $D$
  fixing three of the basis elements $ u_1, u_2, u_3, u_4$
  must fix the fourth one, too. Let $M=C_N(y_0)=C_N(V')$.  It follows
  that $|D_1:M|=4$, so $V'$ is just the homogeneous component of $M$
  corresponding to the trivial representation, while $V''$ is the sum
  of the other homogeneous components of $M$. As $M\nor C_G(y_0)$,
  every element of $C_G(y_0)$ permutes the homogeneous components of
  $M$. Since $y_0\in V'$, we get that $C_G(y_0)$ fixes $V'$, so it also
  fixes the sum of the other homogeneous components of $M$, which is
  $V''$.
\end{proof}
By the previous lemma $C_G(y_0)\cap C_G(x_0+v)$ consists of monomial
matrices for any $v\in V''$, so we can use part \textit{3} of Theorem
\ref{thm:reg_partition} to define a $\pi(C_G(x_0)\cap
C_G(y_0))$-regular partition on $W$.
\begin{thm}\label{thm:req2}
  By part \textit{3} of Theorem \ref{thm:reg_partition} let $W=\{
  e_1\}\cup\{ e_2\}\cup\Omega_3\cup\Omega_4$ be a $\pi(C_G(x_0)\cap
  C_G(y_0))$-regular partition of $W=\{ u_1, u_2,\ldots,
  u_n\}$. Let the vectors $x,y\in V$ be defined as
  \begin{align*}
    &\textrm{For\ \ }|W|\neq 16:&& x=y_0,\ y=x_0+v,\textrm{\qquad where\ \ } 
    v=\!\!\!\!\sum_{ u_i\in\Omega_4\setminus\langle  e_1, e_2\rangle_W}
    \!\!\!\! u_i,\\
    &\textrm{For\ \ }|W|=16:    && x=y_0,\ y=x_0+v,\textrm{\qquad where\ \ } 
    v=0 e_3+2 e_4+
    \!\!\!\!\sum_{\begin{array}{cc}\s 5\leq i\leq 16\\
        \s u_i\notin\{ e_3, e_4\}\end{array}}\!\!\!\! u_i.
  \end{align*}
  Then we have $C_G(x)\cap C_G(y)=1$.
\end{thm}
\begin{proof}
  Let $g\in C_G(x)\cap C_G(y)$. We know by the previous lemma that $g$
  is a monomial matrix fixing all elements of $\langle e_1,
  e_2\rangle_W$.  

  In case of $|W|\neq 16$ also $\Omega_3$ is fixed by
  $\pi(g)$, since $\Omega_3\cup(\Omega_4\setminus\langle e_1,
  e_2\rangle_W)= W\setminus\langle e_1, e_2\rangle_W$, and a basis
  element from $W\setminus\langle e_1, e_2\rangle_W$ occur with
  coefficient $0$ in $v$ if and only if it is an element of
  $\Omega_3$.  It follows that $\pi(g)=1$.  Hence $g=\dg$ is a
  diagonal matrix, and any element in its main diagonal not
  corresponding to $\Omega_3$ must be $1$. Furthermore, since $|W|\neq
  16$, $|\Omega_3|<\frac{1}{4}|W|$ by part \textit{3)} of Theorem
  \ref{thm:reg_partition}, so we get $g=\dg=1$ by using part
  \textit{4} of Theorem \ref{thm:u1_centr}.

  In case of $|W|=16$ we have $\pi(g)( e_3)= e_3$. Now, if
  $\pi(g)( e_4)= e_4$ does not hold, then the number of elements
  in the main diagonal of $\dg$ different from $1$ should be $2$ or
  $3$, which contradicts to part \textit{4} of Theorem
  \ref{thm:u1_centr}.  Hence $\dg=1$ and $\pi(g)( e_4)= e_4$. As
  $ e_1, e_2, e_3, e_4$ is a base for the $\FF 2$-space
  $W$, we get $g=\pi(g)=1$, which proves the identity $C_G(x)\cap
  C_G(y)=1$.
\end{proof}
\subsection{Finding a two-element base in case when 
  \texorpdfstring{$N$}{N} is not monomial}
Now we handle the case when $Z\leq N\leq G\leq GL(U)$ is not monomial.
Since we will use the results of Subsection \ref{subsec:mon}, in order
to be more consistent with our previous notation we denote the whole
vector space by $U$ this time.  Thus, $U$ is an $n=2^k$ dimensional
vector space over $\FF q$, and it is absolutely irreducible as an $\FF
qN$-module.  Furthermore, $q\equiv -1\mod 4$ and $N=Q\star N_1$, where
$Q$ is a quaternion group and $N_1$ is a central product of some
$D_8$'-s and the group of scalar matrices.  Using Lemma
\ref{lem:tensor_direct} it follows that $U= W\otimes V$, for some faithful
absolutely irreducible $\FF qQ$-module $W$ and for some faithful
absolutely irreducible $\FF qN_1$-module $V$. Using that $Q$ has a
unique faithful irreducible representation over $\FF q$, and it is
$2$-dimensional, we have $\dim W=2$ and $\dim V=n/2$.  Let $\{w_1,
w_2\}$ be a basis of $W$ and $V_1=\langle w_1\rangle\otimes V,\
V_2=\langle w_2\rangle\otimes V$.  Then $U=V_1\oplus V_2$ is a direct
decomposition of $U$ into irreducible $\FF qN_1$-modules. Let
$N_G(V_1)$ be the elements of $G$ leaving $V_1$ invariant.  The restriction of
$N_G(V_1)$ to $V_1\simeq V$ defines a homomorphism $N_G(V_1)\ra GL(V)$
with image $G_1\simeq N_G(V_1)/C_G(V_1)$.  Furthermore, $N_1\leq
N_G(V_1)$ is included into $G_1$.  Thus, $N_1\leq G_1\leq GL(V)$ and
we have the monomial case situation.  Using the results of the
previous subsection we have the following
\begin{thm}\label{thm:mon_sum}
  \begin{enumerate}
  \item There is a decomposition $N_1=D\rtimes S$, such that
    $D=Z\times D_1$ and $D_1\simeq S\simeq C_2^{k-1}$. Furthermore,
    $N=(Q\star D)\rtimes S$.  If $g\in QD$ has an eigenvector in $U$,
    then $g\in D$.
  \item There is a basis $\mb {n/2}\in V$ with the following
    properties
    \begin{enumerate}
    \item For each $1\leq i\leq n/2$ let $W_i=W\otimes\langle
      u_i\rangle$.  Then $W_1,W_2,\ldots,W_{n/2}$ are exactly the
      homogeneous components of $QD$ (or $D$ or $D_1$). The elements
      of $D$ act as scalar matrices on each $W_i$. Furthemore,
      $D_1=C_N(W_1)$.
    \item $S$ acts regulary on the sets $\{ w_1\otimes  u_1,
      w_1\otimes  u_2, \ldots, w_1\otimes  u_{n/2}\}$ and
      $\{ w_2\otimes  u_1, w_2\otimes  u_2, \ldots,
      w_2\otimes  u_{n/2}\}$. Thus, $S$ permutes the
      subspaces $W_1,\ldots, W_{n/2}$ regularly.
    \item For any $n\in N$ there is a unique decomposition
      $n=\delta_2(g)\pi(g)$ with $\delta_2(g)\in QD$ and $\pi(g)\in
      S$. Furthermore, $n$ is a monomial matrix if and only if
      $\delta_2(n)\in D$.
    \end{enumerate}
  \item There are two vectors $x_1,y_1\in V$ with the following properties
    \begin{enumerate}
    \item If $g\in N_G(V_1)$ and $g\in C_G( w_1\otimes x_1)\cap 
      C_G( w_1\otimes y_1)$, then $g_{V_1}=\id_{V_1}$.
    \item If $\dim V\geq 4$ then
      $x_1\in \langle  u_1, u_2, u_3, u_4\rangle$ 
      is a linear combination of exactly three basis vectors, while 
      $y_1\in  u_1+\langle  u_5,\ldots,  u_{n/2}\rangle$.
    \item If $\dim V\geq 4$ then $y_1$ contains at least half of 
      the basis vectors $ u_1,\ldots, u_{n/2}$ with 
      non-zero coefficients.
    \end{enumerate}
  \end{enumerate}
\end{thm}
\begin{proof}
  Omitted.
\end{proof}
\begin{thm}\label{thm:req2_notmon}
  Using the vectors $x_1,\ y_1\in V$ as in part \textit{3} of 
  Theorem \ref{thm:mon_sum} let
  \begin{align*}
    x&= w_1\otimes  u_1,&&y= w_2\otimes 
     u_1+ w_1\otimes  u_2,&&\textrm{for }\dim V=2;\\
    x&= w_1\otimes x_1+ w_2\otimes 
     u_1&&y= w_1\otimes y_1,&&\textrm{for }\dim V\geq 4.
  \end{align*}
  Then $C_G(x)\cap C_G(y)=1$.
\end{thm}
\begin{proof}
  In case of $\dim V=2$ let $g\in C_G(x)\cap C_G(y)$. Then $g$
  normalizes $C_N(x)=C_N(W_1)=D_1$, hence it permutes the homogeneous
  components of $D_1$, which are $W_1,W_2$ by part \textit{2.a} of
  Theorem \ref{thm:mon_sum}. It follows that $g$ fixes setwise both
  $W_1$ and $W_2$.  Using that $g$ also fixes $y=w_2\otimes
  u_1+w_1\otimes u_2$ we get that $g$ fixes each of the vectors
  $w_1\otimes u_1, w_2\otimes u_1$ and $w_1\otimes u_2$.  Let $s\in
  S$ be a non-identity element of $S$. Then $[g,s]\in N$, and
  $[g,s]( w_1\otimes u_1)=( w_1\otimes u_1)$. Thus, $[g,s]\in D$ and
  both resctrictions $[g,s]_{W_1}$ and $[g,s]_{W_2}$ are scalar
  matrices.  Then the same holds for $g_{W_1}$ and $g_{W_2}$, so
  $g=1$.

  In case $\dim V= 4$ let $g\in C_G(x)\cap C_G(y)$. Then $g$
  normalizes $C_N(y)=C_N( w_1\otimes  u_1)=D_1$, hence it
  permutes its homogeneous components, that is, $W_1,W_2,W_3,W_4$.  As
  $g(y)=y$, we also have $g(W_1)=W_1$. As $x$ contains $ w_2\otimes
   u_1$ with non-zero coefficient, it follows that $g$ acts
  trivially on $W_1$. Furthermore, $g$ has a decomposition
  $\delta_2(g)\pi(g)$ with $\delta_2(g)$ fixing any of
  $W_1,W_2,W_3,W_4$ and $\pi(g)$ permuting the basis elements.  As in
  part \textit{3} of Theorem \ref{thm:u1_centr}, one can easily see
  that $[\delta_2(g),S]\sbs QD$.  The three basis vectors $
  w_1\otimes  u_i$ contained in $x$ with non-zero coefficients are
  eigenvectors of $\delta_2(g)$, so $[\delta_2(g),s]$ has an
  eigenvector for any $s\in S$, thus $[\delta_2(g),S]\sbs D$ by part
  \textit{1} of Theorem \ref{thm:mon_sum}. As $\delta_2(g)$ acts
  trivially on $W_1$ and $S$ permutes $W_1,W_2,W_3,W_4$ in a
  transitive way, we get that $\delta_2(g)$ acts as a scalar matrix on each
  $W_i$.  Hence $g$ is a monomial matrix, so it fixes $V_1=
  \langle w_1\rangle\otimes V$.  Then $g_{V_1}=\id_{V_1}$ by part \textit{3.a} of
  Theorem \ref{thm:mon_sum}.  Using again that $g=\delta_2(g)$ acts as
  a scalar matrix on each $W_i$, we get $g=1$, as claimed.

  In case of $\dim V>4$ let $U'=W_1\oplus W_2\oplus W_3\oplus W_4$. 
  First we claim that  
  \[
  C_N(x)=\{n\in N\,|\, n \textrm{ acts trivially on }U'\}.
  \]
  As the set of subspaces $W_1,W_2,W_3,W_4$ corresponds to a subspace
  of $S$ (see the paragraph before Lemma \ref{lem:x_three}) and $x$ is
  a linear combination of basis vectors with non-zero coefficients
  from three or four of the subspaces $W_1,W_2,W_3,W_4$, by using a
  similar argument as in Lemma \ref{lem:x_three} we get that $C_N(x)$
  permutes the subspaces $W_1,W_2,W_3,W_4$. Now, for any $n\in C_N(x)$
  there exist $ w_1\otimes  u_s, w_1\otimes  u_t$ occuring
  in $x$ with non-zero coefficients such that $n$ moves $
  w_1\otimes  u_s$ into a multiple of $ w_1\otimes  u_t$, so
  $ w_1\otimes  u_t$ is an eigenvector of $\delta_2(n)\in QD$,
  hence $\delta_2(n)\in D$.  It follows that $C_N(x)=C_N(
  w_1\otimes x_1)\cap C_N( w_2\otimes  u_1)$.  As $C_N(
  w_2\otimes  u_1)=D_1$, we get that $C_N(x)$ consists of diagonal
  matrices, so it fixes every basis element occuring in $x$ with
  non-zero coefficients. Using again that $W_1,W_2,W_3,W_4$
  corresponds to a two dimensional subspace of $S$, we get that $C_N(x)$
  acts trivially on $U'$.  As $x\in U'$, the inclusion $C_N(U')\sbs
  C_N(x)$ holds trivially.

  Let $g\in C_G(x)\cap C_G(y)$. Then $g$ normalizes $C_N(x)$, so it
  fixes the homogeneous component of $C_N(x)$ corresponding to the
  trivial representation of $C_N(x)$, which is exactly $U'$, and the
  sum of the other homogeneous components, that is,
  $W_5\oplus\ldots\oplus W_{n/2}$. Using the form of $y_1$, it follows
  that $g\in C_G( w_1\otimes  u_1)$. Therefore, $g(U')=U'$ and
  $g\in C_G(x)\cap C_G( w_1\otimes  u_1)$.  Then
  $g_{U'}=\id_{U'}$ follows at once from case $\dim V=4$.  As $g\in
  C_G( w_1\otimes  u_1)$, we get that $g$ normalizes $D_1=C_N(
  w_1\otimes  u_1)$, so $g$ permutes $W_1,W_2,\ldots ,
  W_{n/2}$. Let $g=\delta_2(g)\pi(g)$, where $\delta_2(g)$ fixes each
  of the subspaces $W_1,W_2,\ldots,W_{n/2}$, while $\pi(g)$ permutes
  the basis elements.  More than half of the basis vectors $
  w_1\otimes  u_i,\ 1\leq i\leq n/2$ are contained either in $y$,
  or in $U'$, so $\delta_2(g)$ contains eigenvectors from more than
  half of the basis vectors $ w_1\otimes  u_i,\ 1\leq i\leq
  n/2$.  (In fact, here we need to refer to $U'$ only in case $\dim
  V=8$, since otherwise even $y$ contains more than half of the basis
  vectors $ w_1\otimes  u_i$, see Theorem \ref{thm:req2}.)  It
  follows that $[\delta_2(g),s]$ has an eigenvector for any $s\in S$,
  thus $[\delta_2(g),S]\sbs D$.  From this point our proof is the same
  as it was for $\dim V=4$.
\end{proof}
\begin{rem}
  In fact, it can be checked that in this section we used the
  assumption $(|G|,|V|)=1$ only in case of $q=4,\ n=3$ and $q= 3,\
  n=2^k$. Moreover, it is clear that we can extend $x_0,y_0$ with some
  $z_0$ to a three element base of the full group $G$ of symplectic
  type in part \textit{1} of Theorem \ref{thm:GL43}. By extending the
  vectors $x,y$ in Theorems \ref{thm:req2} and \ref{thm:req2_notmon}
  with $z_0$ we get that even if the action is not coprime, a group
  $G\leq GL(2^k,3)$ of symplectic type does always have a base of size
  three.
\end{rem}
\section{Proof of Theorem \ref{thm:main_2base}}
In the following let $\FF q$ be a finite field with $q$ elements of
characteristic $p$. Let $V$ be an $n$ dimensional vector space over
$\FF q$ and $G\leq GL(V)$ be a coprime linear group. Furthermore, let
$Z=Z(GL(V))\simeq \FF q^\t$ denote the group of non-zero scalar
matrices. Then it is clear that $G\leq GZ\leq GL(V)$
and $|GZ|\mid |G|(q-1)$, so $GZ$ is a coprime linear group containing $G$.
Therefore, in order to prove Theorem \ref{thm:main_2base} 
we can (and we will) assume without loss of generality that $G$ contains $Z$. 
\begin{lem}\label{lem:semi}
  Let $G\leq \Gamma L(V)$ be a semilinear group acting on the $\FF
  q$-space $V$ such that $(|G|,|V|)=1$ and let $H=G\cap GL(V)$.  If
  $u_1,u_2\in V$ is a base for $H$, then there exists a $\gamma\in \FF
  q$ such that $u_1,u_2+\gamma u_1$ is a base for $G$.
\end{lem}
\begin{proof}
  For any $g\in G$ let $\sigma_g\in\textrm{Gal}(\FF q|\FF p)$ denote
  the action of $g$ on $\FF q$.  For all $\alpha\in \FF q$ let
  $H_\alpha=C_G(u_1)\cap C_G(u_2+\alpha u_1)\leq G$.  Assuming that
  $C_H(u_1)\cap C_H(u_2)=1$, our goal is to prove that $H_\alpha=1$
  for some $\alpha\in \FF q$.  Let $g\in H_\alpha$. Thus, $g(u_1)=u_1$
  and $u_2+\alpha u_1=g(u_2+\alpha u_1)=
  g(u_2)+\alpha^{\sigma_g}u_1$. Hence
  $g(u_2)=u_2+(\alpha-\alpha^{\sigma_g})u_1$.  If $g\in \langle\cup
    H_\alpha\rangle$, then $g$ is the product of elements from several
  $H_\alpha$'s. It follows that $g(u_2)=u_2+\delta u_1$ for some
  $\delta\in \FF q$.

  We claim that $\langle\cup H_\alpha\rangle\cap H=1.$ Let $g\in
  \langle\cup H_\alpha\rangle\cap H$.  On the one hand, the action of
  $g$ on $V$ is $\FF q$-linear, since $g\in H$.  On the other hand,
  $g(u_1)=u_1$ and $g(u_2)=u_2+\delta u_1$ for some $\delta\in \FF q$
  by the previous paragraph.  If $o(g)=m$, then
  $u_2=g^m(u_2)=u_2+m\delta u_1$, so $m\delta=0$.  Using that $|G|$ is
  coprime to $p$, we get $m$ is not divisible by $p$, hence
  $\delta=0$. Therefore, $g(u_2)=u_2$ and $g\in C_H(u_1)\cap
  C_H(u_2)=1$.

  Let $g,h$ be two distinct elements of $\cup H_\alpha$, so
  $gh^{-1}\not\in H$.  Since $G/H$ is embedded into $\textup{Gal}(\FF
  q|\FF p)$, we get $\sigma_g\neq\sigma_h$. Furthermore, the subfields
  of $\FF q$ fixed by $\sigma_g$ and $\sigma_h$ are the same if and
  only if $\left<g\right>=\left<h\right>$.

  If $g\in H_\alpha\cap H_\beta$, then
  $g(u_2)=u_2+(\alpha-\alpha^{\sigma_g})u_1=u_2+(\beta-\beta^{\sigma_g})u_1$,
  so $\alpha-\beta$ is fixed by $\sigma_g$.  Let $K_g=\{\alpha\in \FF
  q\;|\;g\in H_\alpha\}$.  The previous calculation shows that $K_g$
  is an additive coset of the subfield fixed by $\sigma_g$, so
  $|K_g|=p^d$ for some $d|f=\log_p q$.  Since for any $d|f$ there is a
  unique $p^d$-element subfield of $\FF q$, we get $|K_g|\neq |K_h|$
  unless the subfields fixed by $\sigma_g$ and $\sigma_h$ are the
  same.  As we have seen, this means $\left<g\right>=\left<h\right>$.
  Consequently, $|K_g|\neq |K_h|$ unless $K_g=K_h$.  Hence we get
  \[
  \big|\!\!\!\!\bigcup_{g\,\in\,\cup H_\alpha\setminus\{1\}}\!\!\!\! K_g\big|
  \leq\sum_{d|f,\;d<f}p^d\leq 
  \sum_{d<f}p^d=\frac{p^f-1}{p-1}<p^f=|\FF q|.
  \]
  So there is a $\gamma\in \FF q$ which is not contained in $K_g$ for
  any $g\in\cup H_\alpha\setminus\{1\}$. This exactly means that
  $H_\gamma=C_G(u_1)\cap C_G(u_2+\gamma u_1)=1$.
\end{proof}
\begin{rem} 
  In general, if $G\leq \Gamma L(V)$ is arbitrary and 
  $H=G\cap GL(V)$, then $b(G)\leq b(H)+1$, see \cite[Lemma 3.2]{LS}. 
\end{rem}
Now, we are ready to prove Theorem \ref{thm:main_2base} for arbitrary
coprime linear groups.
\begin{proof}[Proof of Theorem \ref{thm:main_2base}]
  First, we can assume that $G$ is a primitive linear group by Theorem
  \ref{thm:imprim}. 

  Let $N\nor G$ be a normal subgroup of $G$.  Then $V$ is a
  homogeneous $\FF qN$-module, so $V=V_1\oplus V_2\oplus\ldots\oplus
  V_k$, where the $V_i$'s are isomorphic irreducible $\FF qN$-modules.
 
  In case $V_1$ is not absolutely irreducible, let $T\simeq\End_{\FF
    qN}(V_1)$. By Schur's lemma, $T$ is a proper field extension of
  $\FF q$, and
  \[C_{GL(V)}(N)=\End_{\FF qN}(V)\cap GL(V)\simeq GL(k,T).\]
  Furthermore, $L=Z(C_{GL(V)}(N))\simeq Z(GL(k,T))\simeq T^\t$.  Now,
  by using $L$, we can extend $V$ to a $T$-vector space of dimension
  $l:=\dim_T V=\frac{\dim_{\FF q} V}{\dim_{\FF q} T}<\dim_{\FF q} V$.
  As $G\leq N_{GL(V)}(L)$, in this way we get an inclusion $G\leq
  \Gamma L(l,T)$. Using Lemma \ref{lem:semi}, we can assume that
  $G\leq GL(l,T)$.  As $l=\dim_T V<\dim_{\FF q}(V)$, we can use
  induction on the dimension of $V$. Hence in the following we assume
  that $V$ is a direct sum of isomorphic absolutely irreducible $\FF
  qN$-modules for any $N\nor G$.

  Next, let $N\nor G$ and let $V=V_1\oplus\ldots\oplus V_k$ be a
  direct decomposition of $V$ into isomorphic absolutely irreducible
  modules.  By choosing a suitable basis in $V_1,V_2,\ldots, V_k$, we
  can assume that $G\leq GL(n,\FF q)$ such that any element of $N$ is
  of the form $A\otimes I_k$ for some $A\in N_{V_1}\leq GL(n/k,\FF
  q)$.  By using \cite[Lemma 4.4.3(ii)]{Kleidman} we get
  \[N_{GL(n,\FF q)}(N)=\{B\otimes C\,|\,B\in N_{GL(n/k,\FF
    q)}(N_{V_1}),\ C\in GL(k,\FF q)\}.\] Let
  \[
  G_1=\{g_1\in GL(n/k,\FF q)\,|\,\exists g\in G,g_2\in GL(k,\FF q)
  \textrm{ such that }g=g_1\otimes g_2\}.
  \] 
  We define $G_2\leq GL(k,\FF q)$ in an analogous way.  Then $G\leq
  G_1\otimes G_2$. (Here $G_1$ and $G_2$ are not homomorphic images of
  $G$, since $g=g_1\otimes g_2=\lambda\cdot g_1\otimes\lambda^{-1}g_2$
  for any $\lambda\in \FF q^\t$, so the map $g=g_1\otimes g_2\ra g_1$
  is not well-defined.)  However, for any $g=g_1\otimes g_2$ we have
  $g_1^{o(g)}=\lambda\cdot I_{n/k},\ g_2^{o(g)}=\lambda^{-1}\cdot
  I_{k}$ for some $\lambda\in \FF q^\t$, so $o(g_1)\mid|G||\FF q^\t|$
  and $o(g_2)\mid|G||\FF q^\t|$.  It follows that both $(|G_1|,p)=1$
  and $(|G_2|,p)=1$. If $1<k<n$, then by using induction for $G_1\leq
  GL(n/k,\FF q)$ and $G_2\leq GL(k,\FF q)$ we get $b(G_1)\leq 2$ and
  $b(G_2)\leq 2$.  By using Lemma \ref{lem:strong} and \cite[Lemma 3.3
  (ii)]{LS} we have
  \begin{align*}
    b(G)&\leq b(G_1\otimes G_2)=b^*(G_1\otimes G_2)\leq \\
    &\max(b^*(G_1),b^*(G_2))=\max(b(G_1),b(G_2))\leq 2.
  \end{align*}
  Hence in the following we can assume that for any normal subgroup $N\nor G$
  either $N\leq Z=Z(G)\simeq \FF q$ (and $N$ consist of scalar matrices) or 
  $V$ is an absolutely irreducible $\FF qN$-module.

  Let $Z\leq N\nor G$ be such that $N/Z$ is a minimal normal subgroup
  of $G/Z$. Then $N/Z$ is a direct product of isomorphic simple
  groups. If $N/Z$ is a direct product of cyclic groups of prime
  order, then $G$ is of symplectic type. We examined such groups in
  Section \ref{sec:symplectic}, where we proved that $b(G)\leq 2$ for
  such a group.

  In the following let $N/Z$ be a direct product of $t\geq 2$
  isomorphic non-Abelian simple groups. Then $N=L_1\star
  L_2\star\ldots\star L_t$ is a central product of isomorphic groups
  such that for every $1\leq i\leq t$ we have $Z\leq L_i,\ L_i/Z$ is
  simple.  Furthermore, conjugation by elements of $G$ permutes the
  subgroups $L_1,L_2,\ldots, L_t$ in a transitive way.  By choosing an
  irreducible $\FF qL_1$-module $V_1\leq V$, and a set of coset
  representatives $g_1=1,g_2,\ldots,g_t\in G$ of $G_1=N_G(V_1)$ such
  that $L_i=g_iL_1g_i^{-1}$, we get that $V_i:=g_iV_1$ is an
  absolutely irreducible $\FF qL_i$-module for each $1\leq i\leq t$.
  Now, $V\simeq V_1\otimes V_2\otimes\ldots\otimes V_t$ by Lemma
  \ref{lem:tensor_direct} and $G$ permutes the factors of this tensor
  product. It follows that $G$ is embedded into the central wreath
  product $G_1\wr_c S_t$.  If $t\geq 2$, then $b(G_1)\leq 2$ by using
  induction on $\dim V_1$, so $b(G) \leq 2$ by Theorem
  \ref{thm:tensor_main}.

  Finally, let $Z\leq N\nor G$ be such that $N/Z$ is a non-Abelian simple
  group.  Then $N_1=[N,N]\nor G$ is a quasisimple group. Let $\FF p$ be
  the prime field of $\FF q$. Viewing $V$ as an $\FF p G$-module it
  decomposes into the sum of isomorphic irreducible $\FF p
  N_1$-modules.  Let $V_1$ be an irreducible $\FF pN_1$-submodule of
  $V$ and $G_1=\{g\in G\,|\,g(V_1)=V_1\}$ be the stabilizer of $V_1$.
  Then there is a homomorphism $\varphi:G_1\rightarrow GL(V_1)$ which
  is faithful on $N_1$. Therefore, $\ker\varphi\cap N_1=1$, so
  $\ker\varphi\leq C_G(N_1)=C_G(N_1Z)=C_G(N)=Z$.  As any non-identity
  element $g\in Z=\FF q^\t\cdot I$ acts on $V\setminus\{ 0\}$
  fix-point freely, we get $\ker\varphi=1$, so $G_1$ is included in
  $GL(V_1)$. Therefore, we have the situation $N_1\nor G_1\leq
  GL(V_1)$, with $(|G_1|,|V_1|)=1$ and $V_1$ is a vector space over
  the prime field $\FF p$. Moreover, $N_1$ is quasisimple and $V_1$ is
  irreducible as an $\FF p N_1$-module.  By using the results of
  Section \ref{sec:quasisimple} there exist $x,y\in V_1\leq V$ such
  that $C_{G_1}(x)\cap C_{G_1}(y)=1$.  Let $g\in G$ be such that
  $g(x)=x,\ g(y)=y$ and let $N_1x=\{nx\,|\,n\in N_1\}$.  Then
  $gnx=gng^{-1}gx=gng^{-1}x\in N_1 x$ for any $n\in N_1$, so $N_1x$ is
  a $g$-invariant subset.  As the $\FF p$-subspace generated by $N_1
  x$ is exactly $V_1$, we get that $V_1$ is $g$-invariant, that is, $g\in
  G_1$.  This proves that $C_G(x)\cap C_G(y)=C_{G_1}(x)\cap
  C_{G_1}(y)=1$. So $b(G)\leq 2$, which completes our proof.
\end{proof}

\textbf{Acknowledgement.} 
The authors are very grateful to L.~Pyber for taking our attention 
to the papers \cite{goodwin} and \cite{kopa}. Without this information it is
very unlikely that we would have been able to handle the almost quasisimple case.
We are also thank P.~P.~P\'alfy for his helpful comments, especially for providing
us a much simpler proof than our original one for handling Case 3 of
Theorem \ref{thm:quasisimple_main}. We are also very grateful to
A.~Mar\'oti for his many suggestions and for his continuous
encouragement.

\end{document}